\documentclass[11pt]{article}
\usepackage{color}
\usepackage[dvips, outer=2cm, inner=2cm,top=2cm,bottom=2cm]{geometry} 
\usepackage[nottoc]{tocbibind}

\usepackage[parfill]{parskip} 
\usepackage{graphicx}
\usepackage{amsfonts}
\usepackage{amsmath}
\usepackage{amsthm}
\usepackage{amssymb}
\usepackage{epstopdf}
\newtheorem{thm}{Theorem}[section]

\newtheorem{prop}[thm]{Proposition}

\newtheorem{lemma}[thm]{Lemma}

\newtheorem{cor}[thm]{Corollary}

\newtheorem*{thm*}{Theorem}

\DeclareMathOperator{\tr}{tr}
\DeclareMathOperator{\sym}{Sym}

\date{}

\begin{document}

\title{Beyond Endoscopy via the trace formula - III\\
The standard representation}

\author{S. Ali Altu\u{g}}

\maketitle

\abstract
We finalize the analysis of the trace formula initiated in \cite{Altug:2015aa} and developed in \cite{Altug:2015ab}, and calculate the asymptotic expansion of the beyond endoscopic averages for the standard $L$-functions attached to weight $k\geq3$ cusp forms on $GL(2)$ (cf. Theorem \ref{mainthm}). This, in particular, constitutes the first example of beyond endoscopy executed via the Arthur-Selberg trace formula, as originally proposed in \cite{Langlands:2004aa}. As an application we also give a new proof of the analytic continuation of the $L$-function attached to Ramanujan's $\Delta$-function.

\tableofcontents

\section{Introduction}

Let us begin by briefly recalling beyond endoscopy. Since our aim here is to put the current paper in context rather than to give a comprehensive overview, rather than trying to make precise statements we will simply introduce the main problem. For a more through overview we refer the reader to the original article \cite{Langlands:2004aa} and the recent exposition in \cite{Arthur:2015ab}.

Let $G$ be a reductive algebraic group defined over a global field $F$ and let $S$ be a finite set of places of $F$ containing all the archimedean places. The main aim of beyond endoscopy is to isolate those automorphic representations of $G$ (unramified outside of $S$) that are functorial transfers from other groups. The strategy proposed in \cite{Langlands:2004aa} is to use the partial automorphic $L$-functions, $L^S(s,\pi,\rho)$, for various irreducible representations $\rho:{^LG}\rightarrow GL(V)$ to detect functorial transfers. It is based on the expectation that if $L^S(s,\pi,\rho)$ has a pole in $\Re(s)\geq1$ then $\pi$ should be a functorial transfer, and conversely if $\pi$ is a functorial transfer then, by a theorem of Chevalley (cf. \cite{Chevalley:1968aa}), one can find a $\rho$ such that $L^S(s,\pi,\rho)$ has a pole in $\Re(s)\geq1$.

To analyze the poles of
\[L^S(s,\pi,\rho)=\sum_{\substack{n\geq 1}}^{'}\tfrac{a_{\pi,\rho}(n)}{n^s},\]
where the prime indicates that we are summing over $n$ that are not divisible by any of the finite primes in $S$, in the region $\Re(s)\geq1$ one can study the asymptotic behavior, in the variable $X$, of the partial averages
\[\sum_{n<X}^{'}a_{\pi,\rho}(n),\]
up to terms of size $o(X)$. The strategy proposed in \cite{Langlands:2004aa} is to use the trace formula to study the averages above. More precisely, let $\mathbb{A}_S=\prod_{v\in S}\mathbb{Q}_v$, $\mathbb{A}^S=\prod_{v\notin S}^{'}\mathbb{Q}_v$, and $f_S\in C_c^{\infty}(G(\mathbb{A}_S))$. Then, for each $n\geq1$ that is relatively prime to every prime in $S$ there exist a function $f^{n,\rho}\in C_c^{\infty}(G(\mathbb{A}^S))$ such that 
\[\sum_{\pi \in L^2_{disc}(G)}\tr(\pi_S(f_S))a_{\pi,\rho}(n)=\tr(R_{disc}(f_Sf^{n,\rho})),\]
where $L^2_{disc}$ denotes the discrete part of the automorphic spectrum\footnote{We are ignoring the issues about central characters, stability etc. since a thorough discussion of such would take us too far afield.} of $G$. Substituting this expression into the partial averages of coefficients above, we can now state the problem of beyond endoscopy as to study the asymptotic behavior of the averages, 
\begin{equation}\label{be}
\sum_{n<X}^{'}\tr(R_{disc}(f_Sf^{n,\rho})),
\end{equation}
using he Arthur-Selberg trace formula. In \cite{Altug:2015aa} the study of the averages in \eqref{be} was initiated for $G=GL(2)$, $S=\{\infty\}$, and $\rho=\sym^r$, the $r$'th symmetric power representation of ${^L}G=GL(2,\mathbb{C})$.  The geometric side of the trace formula consists of a sum over rational conjugacy classes of (weighted) orbital integrals multiplied by certain arithmetic volume factors. 

The critical part of these sums are the ones over elliptic conjugacy classes, those conjugacy classes whose characteristic polynomials are irreducible over $\mathbb{Q}$. In order to analyze the elliptic part, an appropriate approximate functional equation was introduced in \cite{Altug:2015aa} (cf. (4') of loc. cit.) and a Poisson summation was applied on the so-called Steinberg-Hitchin base (i.e. the space of characteristic polynomials in this case) to isolate the contribution of certain special representations (which, in general, give non-zero contribution to the asymptotic expansion of \eqref{be}, see (31) and (65) of \cite{Langlands:2004aa}) were isolated. In the subsequent paper \cite{Altug:2015ab} the remaining terms after Poisson summation were analyzed giving a firm control over the asymptotic behavior of various Fourier transforms that appear after Poisson summation (cf. theorems A.14 and A.15 of \cite{Altug:2015ab}). We remark that both of the papers cited above are concerned with a single trace formula, in other words they are not concerned with the averages in \eqref{be} but rather prove results for an individual $n$.

The current paper puts all of the previous work together and executes the asymptotic analysis of \eqref{be} for $G=GL(2)$ over $\mathbb{Q}$, $S=\{\infty\}$, and $\rho=\sym^1$, the standard representation of $GL(2,\mathbb{C})$, where we specialize our test function to pick up holomorphic discrete series representation at the archimedean place. This puts us in the framework of classical holomorphic modular forms of weight $k$ and our main theorem can be stated as follows: 
\vspace{0.1in}
\begin{thm}\label{mainthm}Let $k\geq3$ be an integer and for any $n\geq1$ let $T_k(n)$ denote the $n$'th Hecke operator (eigenvalues normalized so that the Ramanujan conjecture reads as $\tr(T_k(n))=O_{k,\epsilon}(n^{\epsilon})$ for every $\epsilon>0$) acting on the space $S_k$ of holomorphic cusp forms of weight $k$. Then, for every $\epsilon>0$
\[\sum_{n<X}\tr(T_k(n))=O_{k,\epsilon}(X^{\frac{31}{32}+\epsilon}).\]

\end{thm}

We will comment on the generality of the methods, the dependence on various assumptions, and possible improvements on the exponent $\frac{31}{32}$ below in the next section. 

To end the introduction we also present an application of the above theorem and give a new proof of the analytic continuation of the $L$-function attached to the Ramanujan $\Delta$-function. Recall that $\Delta(z)$ and $\tau(n)$ are defined by
 \[\Delta(z):=e(z)\prod_{n=1}^{\infty}(1-e(nz))^{24}=\sum_{n=1}^{\infty}\tau(n)e(nz).\]
It is well-known (cf. \cite{Serre:1996aa} pg. 84) that $\Delta(z)$ is a weight $12$ cusp form of full level. Attached to $\Delta$ is the $L$-function
\[L(s,\Delta):=\sum_{n=1}^{\infty}\tfrac{\tau(n)}{n^s}. \]
It is not hard to see that $L(s,\Delta)$ is convergent for $\Re(s)>7$ (loc. cit. (39) on pg. 94). Theorem \ref{mainthm} then has the following immediate corollary
\begin{cor}\label{maincor}$L(s,\Delta)$ extends to a holomorphic function in $\Re(s)>\frac{11}{2}+\frac{31}{32}$.

\end{cor}

Note that because of normalization issues (i.e. with the above definition of $\tau(n)$ the Ramanujan conjecture would read as $\tau(n)=n^{\frac{11}{2}+\epsilon}$ for every $\epsilon>0$) the usual line $\Re(s)=1$ is shifted to $\Re(s)=1+\frac{11}2$. Let us also remark that the novelty of Corollary \ref{maincor} not the result itself, which is well known and dates back to the works of Hecke\footnote{Hecke's proof indeed gives the analytic continuation to the whole complex plane as well as the functional equation for the $L$-finction.} (cf. \cite{Hecke:1959aa}). It is rather the method of proof, which gives the analytic continuation of an automorphic $L$-function using only the trace formula.

\subsection{Several remarks and comments}\label{odds}

\begin{itemize}

\item First, we would like to emphasize that the novelty of the paper is in the method rather than the result. The fact that the standard $L$-functions for $GL(2)$ have analytic continuation is well-known (it goes back to Hecke in the setting of the current paper and is known in much greater generality (cf. \cite{Godement:1972aa})). It is the method by which we prove the analytic continuation that is new and constitutes the first example of beyond endoscopy carried out via the Arthur-Selberg trace formula, following the original proposal in \cite{Langlands:2004aa}.

\item Let us make a couple of remarks on the generality of the method and the various restrictions we have in Theorem \ref{mainthm}. As already mentioned, the specific choice of test functions at the archimedean place puts us in the framework of the classical Selberg trace formula. It also avoids the contribution of the continuous and the non-tempered spectrum (i.e. the trivial representation) with the price that it picks up only holomorphic cusp forms (no Maass forms interfere).

The reason we chose this approach is that it avoids peripheral issues and go directly to the heart of the matter which is the analysis of the averages in \eqref{be} over the elliptic part of the trace formula. Our primary aim in this paper is to analyze the beyond endoscopic averages in \eqref{be} rather than to give a proof of the analytic continuation of the standard $L$-functions for $GL(2)$. The analytic difficulties related to the elliptic part are already present in the Selberg trace formula, and this is why we chose this approach rather than working in the generality of \cite{Altug:2015aa} and \cite{Altug:2015ab}. We should also note the analysis carried out in these references are sufficient to carry out the analysis without any restriction on the archimedean test function.

The assumption about $S=\{\infty\}$ (i.e. we are restricting to representations unramified at every finite place, or classically, forms of full level), on the other hand, is harmless analytically and can be removed without any trouble. It brings in congruence conditions in the sums (depending on the allowed ramification) and does not effect the analysis in any serious way. It just brings extra work on the algebra and would complicate (the already complicated) notation.

\item One can, without much effort, improve the exponent $\frac{31}{32}$ of Theorem \ref{mainthm}. (Our estimates, especially in \S\ref{secell2}, are far from optimal.) Since our aim in this paper is to execute the beyond endoscopic averages in \eqref{be}, which is equivalent to getting an $o(X)$ in Theorem \ref{mainthm}, we did not, in any way, aim for optimality in the exponent.

\item We would also like to briefly mention the connection of the paper with the $\rho$-trace formula of \cite{Arthur:2015ab} (note that $\rho$ is denoted by $r$ in that article). Instead of repeating various definitions and constructions of Arthur we simply refer the reader to \cite{Arthur:2015ab} pages 8-14. The result of Theorem \ref{mainthm}, in particular, proves the $\rho$-trace formula for $G=GL(2)$ and $\rho=$standard representation. In terms of Arthur's notation (cf. (2.3) of loc. cit.) it can be stated as
\[S_{cusp}^{\rho}(f)=0,\]
where $f$ denotes our specific choice of archimedean test function.

\item The last remark is on a peculiar phenomenon about the averages in \eqref{be}. To describe the issue let us first start with the expected result and work our way back. As in Theorem \ref{mainthm} let us fix an integer $k\geq3$. If $\pi$ is an cuspidal automorphic representation of $GL(2)$ attached to a cusp form of weight $k$ of full level, then the standard $L$-function, $L(s,\pi)$, is holomorphic in $\Re(s)\geq1$. Therefore, denoting the $n$'th Hecke operator acting on the space of weight $k$ cusp forms of full level, normalized as in Theorem \ref{mainthm}, by $T_k(n)$, we expect to have
\[\sum_{\pi}Res_{s=1}L(s,\pi)=\lim_{X\rightarrow \infty}\tfrac{1}{X}\sum_{n<X}\tr(T_k(n))=0.\]
This indeed is implied by the main result of the paper, however one sees a surprising feature in the calculation of this limit. Using the trace to calculate the above limit, one gets the limit of the sum\footnote{In what follows one should keep in mind that the contribution of the non-tempered spectrum (the trivial representation in this case), which was shown to contribute to \eqref{sel1} (cf. Theorem 1 of \cite{Altug:2015aa}), is $0$ in this setting, so we can ignore it. In general one needs to subtract it from \eqref{sel1} to make sense of the below claims.} (cf. \S\ref{sectf})   
\[\lim_{X\rightarrow \infty}\tfrac{1}{X}\sum_{n<X}\tr(T_k(n))=\lim_{X\rightarrow \infty}\tfrac{1}{X}\sum_{n<X}\{\eqref{sel1}+\eqref{sel2}+\eqref{sel3}\},\]
where \eqref{sel1} is the contribution of the elliptic conjugacy classes, \eqref{sel2} is the contribution of the hyperbolic and unipotent conjugacy classes (hyperbolic contribution is the sum over all $d\mid n$ that satisfies $d\neq \sqrt{n}$, and the unipotent contribution comes from $d=\sqrt{n}$, which exists only if $n$ is a perfect square), and \eqref{sel3} is the contribution of the identity conjugay class.

The remarkable point is that the limits of the averages of the individual contributions are \emph{non-zero}. More precisely, the proof of Theorem \ref{mainthm} shows that
\begin{equation}\label{weird}
\lim_{X\rightarrow \infty}\tfrac1X\sum_{n<X}\eqref{sel1}=\tfrac{1}{k-1},\hspace{0.2in}\lim_{X\rightarrow \infty}\tfrac1X\sum_{n<X}\eqref{sel2}=\tfrac{1}{1-k},\hspace{0.2in}\lim_{X\rightarrow \infty}\tfrac1X\sum_{n<X}\eqref{sel3}=0.
\end{equation}
So when we sum up the contributions of the terms the result is $0$ however individually the elliptic and hyperbolic contributions\footnote{The proof of Proposition \ref{prophyp} shows that the non-zero contribution is indeed coming from the weighted orbital integrals over hyperbolic conjugacy classes rather than the unipotent ones.} are both non-zero and cancel each other!

There is one more point we would like to highlight. It is not hard to see that in the sums in \eqref{weird} if one, instead of summing over all integers, sums only over primes then the limits of the individual averages are $0$. i.e. 
\begin{equation}\label{weirdp}
\lim_{X\rightarrow \infty}\tfrac1X\sum_{\substack{p<X\\ p:\,prime}}\log(p)(i_p)=\lim_{X\rightarrow \infty}\tfrac1X\sum_{\substack{p<X\\ p:\,prime}}\log(p)(ii_p)=\lim_{X\rightarrow \infty}\tfrac1X\sum_{\substack{p<X\\ p:\,prime}}\log(p)(iii_p)=0.
\end{equation}
One therefore observes that there is a major contribution to the averages of both the hyperbolic and elliptic parts of the trace formula coming from those conjugacy classes with determinant ($=n$) divisible by prime powers, $p^l$, with large $l$, and moreover, these contributions cancel with each other.

An interesting question, also raised by Arthur (cf. Problem VI of \cite{Arthur:2015ab}), is to understand if there is a conceptual explanation for the averages to behave in this way. There is also a further observation one can make on these matters. It was shown in \cite{Altug:2015ab} that the trivial representation and the special representation, denoted by $\tr(\xi_0(f_{\infty}))$, which comes from the continuous spectrum contributes to the term $\xi=0$ after Poisson summation on the Steinberg-Hitchin basis. It turns out that the extra contribution of the weighted orbital integrals of the hyperbolic conjugacy classes also gets cancelled with a part of this term. Is there any connection? (This is, in a sense, combining Problems V and VI of \cite{Arthur:2015ab}). Although the questions are intriguing so far we have no satisfactory explanation for this besides the calculations below.

\end{itemize}

\subsection{Outline of the paper}

The paper is organized as follows: In \S\ref{secresult} we prove Theorem \ref{mainthm} and Corollary \ref{maincor} assuming various estimates on the averages of the geometric side of the trace formula. We hope that this motivates the various estimates to follow in the subsequent sections.

In \S\ref{sectf} we introduce the Selberg trace formula and prove the estimates on the identity (\S\ref{secid}), hyperbolic, and unipotent contributions (\S\ref{sechyp}). 

\S\ref{secell} is the heart of the paper and constitutes the analysis of the elliptic part. Following \cite{Altug:2015aa} we first rewrite the elliptic part and then introduce an approximate functional equation. Then in \S\ref{secell1} we apply Poisson summation on the Steinberg-Hitchin basis (Concretely, the conjugacy classes in $GL(2)$ are parametrized by their characteristic polynomials, and for fixed determinant the only variable is the trace. We apply Poisson summation on this variable, denoted by the variable $m$ in \S\ref{secell}) and analyze the term corresponding to $\xi=0$. In Corollary \ref{cordom} we show that the contribution of this term in the elliptic part exactly cancels (as it should!) the contribution of the hyperbolic part. \S\ref{secell2} forms the main part of the analysis of the elliptic part. In this section we analyze the rest of the terms after Poisson summation. Our strategy\footnote{We remark that applying Poisson summation on the $n$-sum works well for the standard representation and the symmetric square, however stops being productive for higher symmetric powers. This point was observed by Sarnak and we refer to \cite{Sarnak:2001aa} for further discussion of this issue.} is to bring in the sum over $n$ and to use Poisson summation. We do this in two steps. In \S\ref{secell21} we first strip off, for each fixed $n$, certain parts of the elliptic term that does not give any contribution to the asymptotic expansion. Then in \S\ref{secell22} we finally bring the $n$-sum in and use Poisson summation.

In \S\ref{seclocal} we prove certain estimates on Fourier transforms and character sums that are used in the analysis of \S\ref{secell22}.

\subsection{Notation and conventions}

\paragraph{Notation}
\begin{itemize}
\item $\mathbb{Z},\mathbb{R}$ and $\mathbb{C}$ as usual will denote the sets of integers, real, and complex numbers respectively. $\mathbb{N}=\{0,1,2,\cdots\}$ will denote the set of natural numbers. 
\item The Mellin transform of a function $\Phi$ is defined as usual, $\tilde{\Phi}(u)=\int_0^{\infty}\Phi(x)x^{u-1}dx$. 
\item $\sum_{a \bmod^{\times}N}$ will mean $a\in(\mathbb{Z}/N\mathbb{Z})^{\times}$. For an integer $l$, we denote its radical (i.e. the square-free part of it) by $rad(l):=\prod_{\substack{p\mid l}}p$.

\item $\mathcal{S}(\mathbb{R})$ will denote the Schwartz space, $\mathcal{S}(\mathbb{R})=\begin{setdef}{\Phi\in C^{\infty}(\mathbb{R})}{\sup|x^{\alpha}\Phi^{(\beta)}(x)|<\infty\,\,\,\,\forall \alpha,\beta\in\mathbb{N}}\end{setdef}$. We also remark that for functions $\Phi$ which are only defined on $\mathbb{R}^+$, by abuse of notation, we will use $\Phi\in\mathcal{S}(\mathbb{R})$ to mean that $\Phi(x)$ and all of its derivatives decay faster than any polynomial for $x>0$.
\item For a domain $D\subset \mathbb{C}$, $f=O_h(g)$ means that there exists a constant $K$, depending only on $h$, such that $|f(x)|\leq K |g(x)|$ for every $x\in D$. Most of the times $D$ will be clear from the context and we will not be specified. $f\ll_h g$ means $f=O_h(g)$.
\item $\sqrt{\cdot}$ denotes the branch of the square-root function that is positive on $\mathbb{R}^+$, $\left(\tfrac{D}{\cdot}\right)$ denotes the Kronecker symbol, and $e(x):=e^{2\pi i x}$.
\item We also note a slight change of notation from \cite{Altug:2015ab} to the current paper. The function that we denote by $\theta_{\infty}$ in this paper was denoted by $\theta_{\infty,1}^{pos}$ in \cite{Altug:2015ab}. We hope that this simplifies the notation and does not cause any confusion.

\end{itemize}

\paragraph{Conventions}
\begin{itemize}

\item Throughout the paper, unless otherwise explicitly stated (cf. Corollary \ref{maincor}), we always normalize the Hecke operators so that the Ramanujan conjecture is $a_{\pi}(n)=O(n^{\epsilon})$ for every $\epsilon>0$.

\item There is an auxiliary function $F$ that is introduced in the approximate functional equation (cf. \eqref{f}). All of the estimates in \S\ref{secell} depend on the choice of this function. Since this function is fixed once and for all, and does not depend on anything else we will suppress this dependence and not mention it in any of the estimates.

\item Since the function $\theta_{\infty}$ depends only on the weight $k$ (cf. Lemma \ref{ellem1}), instead of $O_{\theta_\infty}(\cdot)$ we will simply write $O_k(\cdot)$.

\end{itemize}

\begin{section}{Proofs of Theorem \ref{mainthm} and Corollary \ref{maincor}}\label{secresult}

 In this section we will give proofs of Theorem \ref{mainthm} and Corollary \ref{maincor} using the results of the rest of the paper so that the reader can follow how each estimate is used.
 
 \begin{proof}[\textbf{Proof of Theorem \ref{mainthm}}] First that if $k$ is odd then the space of cusp forms of full level and weight $k$ is empty so the theorem is trivially true. For the rest of the proof assume $k$ is even. Then,by the Selberg trace formula
 \[\tr(T_k(n))=\eqref{sel1}+\eqref{sel2}+\eqref{sel3}.\]
(For an explanation of the terms involved, and what they actually are, see \S\ref{sectf}.) Substituting this in the average gives
 \begin{align*}
 \sum_{n<X}\tr(T_k(n))&=\sum_{n<X}\{\eqref{sel1}+\eqref{sel2}+\eqref{sel3}\}.
 \end{align*}
By Proposition \ref{propid},
\begin{equation}\label{id}
\sum_{n<X}\eqref{sel3}=O_k(\log(X)).
\end{equation}
By Proposition \ref{prophyp},
\begin{equation}\label{hyp}
\sum_{n<X}\eqref{sel2}=\tfrac{X}{1-k}+O_k(\sqrt X).
\end{equation}
To bound the average of \eqref{sel1} we use corollaries \ref{cordom}, \ref{cors0}, \ref{cors1}, \ref{cors2}, and \ref{cors}. For any $\kappa,\alpha>0$ such that $2\kappa+\alpha<\frac{1}{12}$, and every $N>0$  these give
 \begin{align*}
\sum_{n<X}\eqref{sel1}&= \tfrac{X}{k-1}+O(X^{2-N\kappa}+X^{\frac{3+2\kappa}{4}}\log^2(X)+X^{1-\frac{3\alpha}{2}}\log(X)+X^{1-\kappa}\log^2(X)+X^{\frac{11}{12}+\kappa+\alpha+\frac{11\epsilon}{6}}),
 \end{align*}
where the implied constant above depends only on $k,N,\kappa, \alpha$, and $ \epsilon$. Choosing $\kappa=\frac{1}{32}-\frac{\epsilon}{2},\alpha=\frac{1}{48}-\frac{\epsilon}{3}$, and $N=\frac2\kappa$ then gives,
\begin{equation}\label{ell}
\sum_{n<X}\eqref{sel1}=O(X^{\frac{31}{32}+\epsilon}).
\end{equation}
The theorem follows from \eqref{id}, \eqref{hyp}, and \eqref{ell}.
 \end{proof}

 \begin{proof}[\textbf{Proof of Corollary \ref{maincor}.}] The proof is a straightforward application of partial summation. For any $N_1>N_0$ we get
\begin{equation} \label{sumbyparts}
\sum_{n=N_0}^{N_1}\tfrac{\tau(n)}{n^s}=\tfrac{\tau(N_1)}{N_1^s}-\sum_{n=N_0}^{N_1-1}\Bigl(\frac{1}{n^{s}}-\frac{1}{(n-1)^s}\Bigr)B(n),
\end{equation}
where $B(n):=\sum_{m=N_0}^n\tau(m)$. It is well-known that $ \Delta(z)$ is a cuspidal Hecke eigenform of weight $12$ (cf. pg. 84 of \cite{Serre:1996aa}), and that the space of cusp forms of weight $12$ for the full modular group has dimension $1$ (cf. pg. 96 of loc. cit.). Keeping in mind the normalization we have in Theorem \ref{mainthm} we therefore have $\tau(n)=n^{\frac{11}{2}}\tr(T_{12}(n))$ for every $n\geq1$. Hence, by Theorem \ref{mainthm} for every $\epsilon>0$ we have
\[B(N_0,n)=O(n^{\frac{11}2+\frac{31}{32}+\epsilon}).\]
Substituting these bounds in \eqref{sumbyparts} then gives
\[\sum_{n=N_0}^{N_1}\tfrac{\tau(n)}{n^s}\ll \frac{1}{N_0^{s-(\frac{11}{2}+\frac{31}{32}+\epsilon)}}.\]
 Therefore, for $\Re(s)>\frac{11}{2}+\frac{31}{32}$ the partial sums converge uniformly to $L(s,\Delta)$ and the corollary follows.
 
 \end{proof}

\end{section}

\begin{section}{The trace fromula}\label{sectf}

Let $k>2$ be an \emph{even} integer. Recall (cf.\cite{Selberg:1956aa} or \cite{A.-Knightly:2006aa}) that the Selberg trace formula expresses the trace of the $n$'th Hecke operator acting on cusp forms of weight $k$ and full level as the sum
 \begin{align*}
  \tr(T_k(n))&= -\tfrac{1}{2n^{\frac{k-1}{2}}}\sum_{\substack{m<\sqrt{n}}}\tfrac{\rho^{k-1}-\bar{\rho}^{k-1}}{\rho-\bar{\rho}} \sum_{\substack{f^2\mid (4n-m^2)\\ \frac{m^2-4n}{f^2}\equiv 0,1\bmod 4}}h_w\left(\tfrac{m^2-4n}{f^2}\right)\tag{$i_n$}\label{sel1}\\
 &-\tfrac12\sum_{d\mid n}\min\left(\tfrac{d}{\sqrt{n}},\tfrac{\sqrt{n}}{d}\right)^{k-1}.\tag{$ii_n$}\label{sel2}\\
  &+ \tfrac{k-1}{12\sqrt{n}}\delta_\square(n)\tag{$iii_n$}\label{sel3}
 \end{align*}
where
\begin{equation}\label{eqrho}
\rho=\tfrac{m+\sqrt{m^2-4n}}{2}\hspace{0.5in},\hspace{0.5in} \bar{\rho}=\tfrac{m-\sqrt{m^2-4n}}{2},
\end{equation}
and
\[\delta_\square(n)=\begin{cases}1&\text{if $n$ is a square}\\ 0&\text{otherwise}\end{cases}.\]
 The function $h_w(\alpha)$ is defined as the class number of the order of discriminant $\alpha$ weighted by $1/2$ or $1/3$ if $\alpha=-4$ or $\alpha=-3$ respectively.
 
Note that the classical Selberg trace formula, as stated in \cite{Selberg:1956aa}, is $n^{\frac{k-1}{2}}$ times the one given above. This is, once again, due to the normalization of the eigenvalues of the Hecke operators. We normalize them so that the Ramanujan conjecture reads as $|a_\pi(p)|\leq2$.

We also remark that in the above expression, \eqref{sel3} corresponds to the contribution of the conjugacy class of the identity element, \eqref{sel1} corresponds to the contribution of conjugacy classes of elliptic elements, and \eqref{sel2} is the contribution of the sum of the contribution of the hyperbolic conjugacy classes (terms for which $n\neq d^2$) and unipotent conjugacy classes (appears only if $\delta_{\square}(n)=1,$ and is the term corresponding to $d^2=n$).

\subsection{Contribution of the conjugacy class of the identity element}\label{secid}

\begin{prop}\label{propid}
\[\tfrac{k-1}{12}\sum_{n<X}\tfrac{\delta_\square(n)}{\sqrt{n}}=O_k\left(\log(X)\right).\]

\end{prop}

\begin{proof}
Obvious.

\end{proof}

\subsection{Contribution of the hyperbolic and unipotent conjugacy classes}\label{sechyp}

As explained in \S\ref{odds}, quite interestingly, the hyperbolic conjugacy classes do contribute to the limit.

\begin{prop}\label{prophyp}
\[-\tfrac12 \sum_{n<X}\sum_{d\mid n}\min\left(\tfrac{d}{\sqrt{n}},\tfrac{\sqrt{n}}{d}\right)^{k-1}=\tfrac{X}{1-k}+O_k(\sqrt X).\]
\end{prop}

\begin{proof} First note that for any $d$ dividing $n$ such that $d\neq \sqrt{n},$ $d$ and $\tfrac{n}{d}$ give the same contribution to the inner sum. Therefore,
\[\sum_{d\mid n}\min\left(\tfrac{d}{\sqrt{n}},\tfrac{\sqrt{n}}{d}\right)^{k-1}=2\sum_{\substack{d\mid n\\ d<\sqrt{n}}}\left(\tfrac{d}{\sqrt{n}}\right)^{k-1}+\delta_{\square}(n).\]
Next, we trivially have
\[\sum_{n<X}\delta_\square(n)=O(\sqrt{X}).\tag{$*$}\label{prop2star}\]
For the rest of the sum, first note that
\begin{align*}
\sum_{n<X}\sum_{\substack{d\mid n\\ d<\sqrt{n}}}\left(\tfrac{d}{\sqrt{n}}\right)^{k-1}&=\sum_{\substack{ab<X\\ a<b}}\left(\tfrac{a}{b}\right)^{\frac{k-1}{2}}=\sum_{\substack{b\leq\sqrt{X}\\ a<b}}\left(\tfrac{a}{b}\right)^{\frac{k-1}{2}}+\sum_{\substack{\sqrt{X}<b<X\\ a<\frac Xb}}\left(\tfrac{a}{b}\right)^{\frac{k-1}{2}}.\tag{$**$}\label{prop2starstar}
\end{align*}
Then, by Euler-Maclaurin formula
\begin{align*}
\sum_{a<b}a^{\frac{k-1}{2}}&=\tfrac{2}{k+1}b^{\frac{k+1}{2}}+O_{k}\bigl(b^{\frac{k-1}{2}}\bigr)\hspace{0.3in}\text{and}\hspace{0.3in} \sum_{a<\frac Xb}a^{\frac{k-1}{2}}=\tfrac{2}{k+1}\left(\tfrac{X}{b}\right)^{\frac{k+1}{2}}+O_k\bigl(\left(\tfrac{X}{b}\right)^{\frac{k-1}{2}}\bigr).
\end{align*}
Therefore, 
\begin{equation*}
\sum_{\substack{b\leq\sqrt{X}\\ a<b}}\left(\tfrac{a}{b}\right)^{\frac{k-1}{2}}=\tfrac{X}{k+1}+O_k(\sqrt{X})\hspace{0.3in}\text{and}\hspace{0.3in}
\sum_{\substack{\sqrt{X}<b<X\\ a<\frac Xb}}\left(\tfrac{a}{b}\right)^{\frac{k-1}{2}}=\tfrac{2X}{(k+1)(k-1)}+O_k(\sqrt{X}).
\end{equation*}
Substituting the above estimates in \eqref{prop2starstar} and combining the result with \eqref{prop2star} finishes the proof.

\end{proof}

\section{Contribution of the elliptic conjugacy classes}\label{secell}

Following \cite{Altug:2015aa} we begin the analysis by rewriting the elliptic part so that it becomes more suitable form.  

\begin{lemma}\label{ellem1}Let $m,n\in\mathbb{N}$ be fixed and $\rho,\bar{\rho}$ be defined by \eqref{eqrho}. Then,
\[-\tfrac{1}{2n^{\frac{k-1}{2}}}\tfrac{\rho^{k-1}-\bar{\rho}^{k-1}}{\rho-\bar{\rho}}=\tfrac{\pi i}{\sqrt{m^2-4n}}\theta_{\infty}\left(\tfrac{m}{2\sqrt{n}}\right),\]
where
\[\theta_{\infty}(x):=\begin{cases}\tfrac{i}{2\pi}\left\{(x+\sqrt{x^2-1})^{k-1}-(x-\sqrt{x^2-1})^{k-1}\right\}& |x|<1\\0&otherwise\end{cases}.\]
\end{lemma}

\begin{proof}Obvious.

\end{proof}

\begin{lemma}\label{ellem2}Let $m,n\in\mathbb{N}$ be such that $m^2-4n<0$ and $h_w$ be defined as in \S\ref{secresult}. Then,
\[ \sum_{\substack{f^2\mid (4n-m^2)\\ \frac{m^2-4n}{f^2}\equiv 0,1\bmod 4}}h_w\left(\tfrac{m^2-4n}{f^2}\right)=\tfrac{\sqrt{4n-m^2}}{\pi }L(1,m^2-4n),\]
where $L(s,m^2-4n)$ is the following weighted sum of quadratic Dirichlet $L$-functions 
\[L(s,m^2-4n):=\sum_{\substack{f^2\mid m^2-4n\\ \frac{m^2-4n}{f^2}\equiv 0,1\bmod4}}\tfrac{1}{f^{2s-1}}L\left(s,\left(\tfrac{(m^2-4n)/f^2}{\cdot}\right)\right).\]

\end{lemma}

\begin{proof}Recall that for a fundamental discriminant $D<0$ the class number formula states
\[L\left(1,\left(\tfrac{D}{\cdot}\right)\right)=\tfrac{2\pi h(D)}{w_D\sqrt{|D|}},\]
where $w_D$ is the number of roots of unity in $\mathbb{Q}(\sqrt{D})$ and $h(D)$ is the class number of the same field.

Let $m^2-4n=D(m,n)s(m,n)^2,$ where $D(m,n)$ is a fundamental discriminant (i.e. the discriminant of the field $\mathbb{Q}(\sqrt{m^2-4n})$). Then for any $f\mid s(m,n),$ by Theorem 7.24 of \cite{Cox:1989aa} we have
\[h_w\left(\tfrac{m^2-4n}{f^2}\right)=\tfrac{s(m,n)}{f}\tfrac{2 h(D(m,n))}{w_{D(m,n)}}\prod_{q\mid \frac{s(m,n)}{f}}\left(1-\left(\tfrac{D(m,n)}{q}\right)\tfrac{1}{q}\right).\]
Moreover the condition that $f^2\mid m^2-4n$ such that $\tfrac{m^2-4n}{f^2}\equiv 0,1\bmod 4$ is equivalent to $f\mid s(m,n)$. Therefore,
\begin{align*}
\sum_{\substack{f^2\mid m^2-4n\\ \frac{m^2-4n}{f^2}\equiv 0,1\bmod 4}}h_w\left(\tfrac{m^2-4n}{f^2}\right)&=\tfrac{2s(m,n)h(D(m,n))}{w_{D(m,n)}}\sum_{f\mid s(m,n)}\tfrac{1}{f}\prod_{q\mid \frac{s(m,n)}{f}}\left(1-\left(\tfrac{D(m,n)}{q}\right)\tfrac{1}{q}\right)\\
&=\tfrac{\sqrt{4n-m^2}}{\pi}L(1,m^2-4n).
\end{align*}
\end{proof}

\begin{cor}The elliptic part of the trace formula can be expressed as
\[\sum_{m\in\mathbb{Z}}\theta_{\infty}\left(\tfrac{m}{2\sqrt{n}}\right)L(1,m^2-4n),\]
where $\theta_{\infty}(x)$ and $L(1,m^2-4n)$ are defined as in lemmas \ref{ellem1} and \ref{ellem2} respectively.

\end{cor}
\begin{proof}Follows from Lemma \ref{ellem2}.

\end{proof}

\subsection{Approximate functional equation, Poisson summation, and the contribution of the dominant term}\label{secell1}

Next, let us recall the approximate functional equation introduced in \cite{Altug:2015aa}. Let\footnote{We remark that the specific choice of this function is irrelevant to the rest of the argument. Only the pole of its Mellin transform is important.} $F(x)$ be
\begin{equation}\label{f}
F(x)=\tfrac{1}{2K_0(2)}\int_{x}^{\infty}e^{-y-\tfrac{1}{y}}\tfrac{dy}{y}.
\end{equation} 
The only property of the above function that we will use (cf. \cite{Iwaniec:2004aa} pg. 257-258) is the following: 

\begin{center}\emph{The Mellin transform, $\tilde{F}(u)$, is holomorphic except for a simple pole at $u=0$ with residue $1$.}\end{center}

Let $m<2\sqrt{n}$. Then, by Corollary 3.5 of loc. cit., where we have taken $A=4n-m^2$ and $\iota_{\delta}=1$, we have
\[L(1,m^2-4n)=\sum_{\substack{f^2\mid m^2-4n\\ \frac{m^2-4n}{f^2}\equiv 0,1 \bmod4}}\tfrac{1}{f}\sum_{l=1}^{\infty}\tfrac{1}{l}\left(\tfrac{(m^2-4n)/f^2}{l}\right)\left[F\left(\tfrac{lf^2}{4n-m^2}\right)+\tfrac{ lf^2}{\sqrt{4n-m^2}}H\left(\tfrac{lf^2}{4n-m^2}\right)\right],\]
where
\begin{equation}\label{h}
H(y):=\tfrac{\sqrt{\pi}}{2\pi i}\int_{(1)}\tfrac{\Gamma\left(\frac{1+u}{2}\right)}{\Gamma\left(\frac{2-u}{2}\right)}(\pi y)^{-u}\tilde{F}(u)du.
\end{equation}

\subsubsection{Poisson summation}

We begin with introducing a shorthand notation for the Fourier transform that will frequently appear throughout the text.
\begin{equation*}
I_{l,f}(\xi,n):=\int_{-1}^1\theta_{\infty}(x)\left\{F\left(\tfrac{lf^2(4n)^{-1/2}}{\sqrt{1-x^2}}\right)+\tfrac{lf^2(4n)^{-1/2}}{\sqrt{1-x^2}}H\left(\tfrac{lf^2(4n)^{-1/2}}{\sqrt{1-x^2}}\right)\right\}e\left(\tfrac{-x\xi \sqrt{n}}{2lf^2}\right)dx.
\end{equation*}

\begin{lemma}\label{lempois}
\begin{equation}\label{afterpois}
\sum_{m\in\mathbb{Z}}\theta_{\infty}\left(\tfrac{m}{2\sqrt{n}}\right)L(1,m^2-4n)=\tfrac{\sqrt{n}}{2}\sum_{f,l=1}^{\infty}\tfrac{1}{(lf^2)^{\frac32}}\sum_{\xi\in\mathbb{Z}}\tfrac{Kl_{l,f}(\xi,n)}{\sqrt{l}}I_{l,f}(\xi,n)
\end{equation}

\end{lemma}
\begin{proof}
This is just a restatement of theorem 4.2 of  \cite{Altug:2015aa}, where we take $\alpha=\frac12$.
\end{proof}

\subsubsection{Analysis of the term $\xi=0$}

\begin{lemma}\label{auxlem}Let $\beta,z\in\mathbb{C}$ be such that $z>1$ and $\beta-\frac z2>1$. Then,

\[\sum_{n=1}^{\infty}\tfrac{1}{n^{\beta-\frac{z}{2}}}\sum_{f,l=1}^{\infty}\tfrac{1}{(lf^2)^{z+\frac12}}\tfrac{Kl_{l,f}(0,n)}{\sqrt{l}}=4\tfrac{\zeta(2z)\zeta\left(\beta+\frac z2\right)\zeta\left(\beta-\frac z2\right)}{\zeta(z+1)}\]
\end{lemma}

\begin{proof}By Corollary B.8 of \cite{Altug:2015ab} $Kl_{l,f}(0,n)=O(\log(lf^2)l\gcd(n,f^2))$, hence the left hand side is absolutely convergent. By Corollary 5.4 of  \cite{Altug:2015aa} we have 
\[\sum_{f,l=1}^{\infty}\tfrac{1}{(lf^2)^{z+\frac12}}\tfrac{Kl_{l,f}(0,n)}{\sqrt{l}}=4\tfrac{\zeta(2z)}{\zeta(z+1)}\prod_{p\mid n}\tfrac{1-p^{-z(v_p(n)+1)}}{1-p^{-z}}.\]
Summing this over $n$ gives
\begin{align*}
4\tfrac{\zeta(2z)}{\zeta(z+1)}\sum_{n=1}^{\infty}\tfrac{1}{n^{\beta-\frac{z}{2}}}\prod_{p\mid n}\tfrac{1-p^{-z(v_p(n)+1)}}{1-p^{-z}}&=4\tfrac{\zeta(2z)\zeta(z)}{\zeta(z+1)}\prod_p\left(\sum_{m=0}^{\infty}\tfrac{1-p^{-z(m+1)}}{p^{m(\beta-\frac z2)}}\right)\\
&=4\tfrac{\zeta(2z)\zeta\left(\beta+\frac z2\right)\zeta\left(\beta-\frac z2\right)}{\zeta(z+1)}.
\end{align*}
\end{proof}

\begin{prop}\label{propdom1} For $k\geq3$ we have,

\begin{align*}
\tfrac12\sum_{n<X}\sqrt{n}\sum_{f,l=1}^{\infty}\tfrac{1}{(lf^2)^{\frac32}}\tfrac{Kl_{l,f}(0,n)}{\sqrt{l}}\int_{-1}^{1}\theta_{\infty}\left(x\right)F\left(\tfrac{lf^2(4n)^{-1/2}}{\sqrt{1-x^2}}\right)dx&\\
&\hspace{-3in}=2X^{\frac32}\zeta(2)\int\theta_{\infty}(x)dx+2^{\frac{-1}{2}}X^{\frac54}\tilde{F}\left(\tfrac{-1}{2}\right)\int \tfrac{\theta_{\infty}(x)}{(1-x^2)^{\frac14}}dx+O_{k}(1).
\end{align*}
We remark that the first term above, which is of size $X^{\frac32},$ is the contribution of the trivial representation, cf. lemma 6.2 of  \cite{Altug:2015aa} (which, in fact, is $0$ in our case. cf. corollary \ref{cordom}).
\end{prop}

\begin{proof}The proof is a simple application of Mellin inversion and contour shifting. By Perron's formula
\[\tfrac{1}{4\pi i}\int_{(4)}\tfrac{X^u}{u}\left[\sum_{n=1}^{\infty}\tfrac{1}{n^{u-\frac12}}\sum_{f,l=1}^{\infty}\tfrac{1}{(lf^2)^{\frac{3}{2}}}\tfrac{Kl_{l,f}(0,n)}{\sqrt{l}}\int_{-1}^{1}\theta_{\infty}\left(x\right)F\left(\tfrac{lf^2(4n)^{-1/2}}{\sqrt{1-x^2}}\right)dx\right]du.\tag{$*$}\label{prop2.8*}\]
We note that the absolute convergence of the triple sum for $\Re(u)>2$ is guaranteed by the estimate $\sum_{l,f}\cdots=O_{\theta_{\infty},F}(\sqrt{n}),$ which follows from theorem 6.1 of  \cite{Altug:2015aa}. Next, we use Mellin inversion on $F$ which gives
\[\eqref{prop2.8*}=\tfrac{1}{2(2\pi i)^2}\int_{(4)}\int_{(2)}\tfrac{X^u\tilde{F}(w)2^w}{u}\left[\sum_{n=1}^{\infty}\tfrac{1}{n^{u-\frac{1+w}{2}}}\sum_{f,l=1}^{\infty}\tfrac{1}{(lf^2)^{w+\frac32}}\tfrac{Kl_{l,f}(0,n)}{\sqrt{l}}\int_{-1}^{1}(1-x^2)^{\frac w2}\theta_{\infty}\left(x\right)dx\right]dwdu.\]
We note that the same bound as above guarantees the absolute convergence of the triple sum since $\Re(u-\frac{1+w}{2})=\frac52>\frac32$. Using Lemma \ref{auxlem}, with $\beta=u$ and $z=w+1,$ in the inner sums gives
\[\eqref{prop2.8*}=\tfrac{2}{(2\pi i)^2}\int_{(4)}\int_{(2)}\tfrac{X^u\tilde{F}(w)2^w}{u}\left[\tfrac{\zeta(2(w+1))\zeta\left(u+\frac{w+1}{2}\right)\zeta\left(u-\frac{w+1}{2}\right)}{\zeta(w+2)}\int_{-1}^{1}(1-x^2)^{\frac w2}\theta_{\infty}\left(x\right)dx\right]dwdu.\]
The rest of the proof is contour shifting. We will shift the $w$-contour to right and the $u$-contour to left. To shift the contours first note that the only pole of 
\[\tfrac{\zeta(2(w+1))\zeta\left(u+\frac{w+1}{2}\right)\zeta\left(u-\frac{w+1}{2}\right)}{\zeta(w+2)}\]
in the region $\Re(w)\geq2$ is simple and is at $w=2u-3$ with residue $-2\zeta(4u-4)$ and all the other terms depending on $w$ are holomorphic in $\Re(w)\geq2$. Therefore, moving the $w$-contour to $\Re(w)=6$ gives\footnote{We remark that since we are moving the contour from left to right the residue formula has an extra $``-"$ sign. }
\begin{align*}
\eqref{prop2.8*}&=\tfrac{4}{2\pi i}\int_{(4)}\tfrac{X^u\tilde{F}(2u-3)2^{2u-3}}{u}\zeta(4u-4)\left[\int_{-1}^{1}(1-x^2)^{u-\frac32}\theta_{\infty}\left(x\right)dx\right]du\\
&+\tfrac{2}{(2\pi i)^2}\int_{(4)}\int_{(6)}\tfrac{X^u\tilde{F}(w)2^w}{u}\left[\tfrac{\zeta(2(w+1))\zeta\left(u+\frac{w+1}{2}\right)\zeta\left(u-\frac{w+1}{2}\right)}{\zeta(w+2)}\int_{-1}^{1}(1-x^2)^{\frac w2}\theta_{\infty}\left(x\right)dx\right]dwdu.
\end{align*}
We first handle the second integral above. Since $F\in C_c^{\infty}(\mathbb{R}_+)$, $\tilde{F}(w)$ is rapidly decreasing on the line $\Re(w)=6$. Moreover the ratio of $\zeta$-functions is rapidly decreasing in vertical strips and hence we can interchange the $u$ and $w$ integrals. Moreover the $u$-integrand is holomorphic in $4\geq\Re(u)\geq-1$, except with a simple pole at $u=0$. Hence by interchanging the $u$ and $w$ integrals and moving the $u$-contour to $\Re(u)=-\frac{1}{2}$ (which picks up the residue at $u=0$) we get that the second integral is $O_{k}(1)$.

For the first integral, we again move the $u$-contour to $\Re(u)=-\frac12$ (Note that the $x$-integral still converges since we are assuming that the weight, $k$, of the forms we are working with satisfies $k\geq3$ which implies $\theta_{\infty}(x)$ vanishes to order at least $2$ at $x=\pm1,$ hence $(1-x^2)^{-3/2}\theta_{\infty}(x)$ is integrable around $x=\pm1$.). This picks up the pole of $\zeta(4u-4)$ at $u=\frac54$ with residue $\frac14$, the pole of $\tilde{F}(2u-3)$ at $u=\frac23$ with residue $\frac12$, and the pole of $\frac{1}{u}$ at $u=0$ with residue $1$. Therefore, 
\begin{align*}
\eqref{prop2.8*}&=2X^{\frac32}\zeta(2)\int\theta_{\infty}(x)dx+2^{\frac{-1}{2}}X^{\frac54}\tilde{F}\left(\tfrac{-1}{2}\right)\int \tfrac{\theta_{\infty}(x)}{(1-x^2)^{\frac14}}dx+O_{k}(1).
\end{align*}
Combining the estimates above for the two integrals finishes the proof.

\end{proof}

\begin{prop}\label{propdom2}For $k\geq3$ we have,

\begin{align*}
\tfrac14\sum_{n<X}\sum_{f,l=1}^{\infty}\tfrac{1}{\sqrt{lf^2}}\tfrac{Kl_{l,f}(0,n)}{\sqrt{l}}\int_{-1}^{1}\tfrac{\theta_{\infty}\left(x\right)}{\sqrt{1-x^2}}H\left(\tfrac{lf^2(4n)^{-1/2}}{\sqrt{1-x^2}}\right)dx&\\
&\hspace{-2.5in}=2^{\frac{-1}{2}}X^{\frac54}\tilde{F}\left(\tfrac12\right)\int \tfrac{\theta_{\infty}(x)}{(1-x^2)^{\frac14}}dx+\pi X\zeta(0)\int\tfrac{\theta_{\infty}(x)}{\sqrt{1-x^2}}dx+O_{k}(1).
\end{align*}

\end{prop}

\begin{proof}The proof of Proposition \ref{propdom1} goes verbatim. We give the details for completeness. By Perron's formula
\[\tfrac{1}{8\pi i}\int_{(4)}\tfrac{X^u}{u}\left[\sum_{n=1}^{\infty}\tfrac{1}{n^u}\sum_{f,l=1}^{\infty}\tfrac{1}{\sqrt{lf^2}}\tfrac{Kl_{l,f}(0,n)}{\sqrt{l}}\int_{-1}^{1}\tfrac{\theta_{\infty}\left(x\right)}{\sqrt{1-x^2}}H\left(\tfrac{lf^2(4n)^{-1/2}}{\sqrt{1-x^2}}\right)dx\right]du.\tag{$*$}\label{prop2.9*}\]
Once again we note that the absolute convergence of the triple sum and the interchange of the integrals is guaranteed by the estimate $\sum_{l,f}\cdots=O_{\theta_{\infty},F}(\sqrt{n}),$ which follows from theorem 6.1 of  \cite{Altug:2015aa}. Next, substituting \eqref{h} for the function $H$ we get that \eqref{prop2.9*} is
\[\tfrac{\sqrt{\pi}}{4(2\pi i)^2}\int_{(4)}\int_{(1)}\tfrac{X^u\tilde{F}(w)}{u}\tfrac{2^w\Gamma\left(\frac{1+w}{2}\right)}{\pi^{w}\Gamma\left(\frac{2-w}{2}\right)}\left[\sum_{n=1}^{\infty}\tfrac{1}{n^{u-\frac w2}}\sum_{f,l=1}^{\infty}\tfrac{1}{(lf^2)^{w+\frac12}}\tfrac{Kl_{l,f}(0,n)}{\sqrt{l}}\int_{-1}^{1}(1-x^2)^{\frac{w-1}{2}}\theta_{\infty}(x)dx\right]dwdu,\]
where the interchange of the integrals with the triple sum is justified by the same bound above using $\Re(u-\frac w2)=3>\frac{3}{2}$. By lemma \ref{auxlem} with $\beta=u$ and $z=w$ we have
\[\eqref{prop2.9*}=\tfrac{\sqrt{\pi}}{(2\pi i)^2}\int_{(4)}\int_{(1)}\tfrac{X^u\tilde{F}(w)}{u}\tfrac{2^w\Gamma\left(\frac{1+w}{2}\right)}{\pi^{w}\Gamma\left(\frac{2-w}{2}\right)}\left[\tfrac{\zeta(2w)\zeta\left(u+\frac{w}{2}\right)\zeta\left(u-\frac{w}{2}\right)}{\zeta(w+1)}\int_{-1}^{1}(1-x^2)^{\frac{w-1}{2}}\theta_{\infty}(x)dx\right]dwdu\]
We now finish the proof by shifting contours. We will shift the $w$-contour to right and the $u$-contour to left. To shift the contours first note that the only pole of 
\[\tfrac{\Gamma\left(\frac{1+w}{2}\right)}{\Gamma\left(\frac{2-w}{2}\right)}\tfrac{\zeta(2w)\zeta\left(u+\frac{w}{2}\right)\zeta\left(u-\frac{w}{2}\right)}{\zeta(w+1)}\]
in the region $\Re(w)\geq1$ is at $w=2u-2$ with residue $\frac{-2\Gamma\left(u-\frac12\right)\zeta(4u-4)}{\Gamma(2-u)}$ and all the other terms depending on $w$ are holomorphic in $\Re(w)\geq1$. Therefore, moving the $w$-contour to $\Re(w)=6$ gives\footnote{As in Proposition \ref{propdom1} since we are moving the contour from left to right the residue formula has an extra $``-"$ sign. }
\begin{align*}
\eqref{prop2.8*}&=\tfrac{2\sqrt{\pi}}{2\pi i}\int_{(4)}\tfrac{X^u\tilde{F}(2u-2)}{u}\tfrac{2^{2u-2}\Gamma\left(u-\frac12\right)}{\pi^{2u-2}\Gamma(2-u)}\zeta(4u-4)\left[\int_{-1}^{1}(1-x^2)^{u-\frac32}\theta_{\infty}\left(x\right)dx\right]du\\
&+\tfrac{2}{(2\pi i)^2}\int_{(4)}\int_{(6)}\tfrac{X^u\tilde{F}(w)}{u}\tfrac{2^w\Gamma\left(\frac{1+w}{2}\right)}{\pi^{w}\Gamma\left(\frac{2-w}{2}\right)}\left[\tfrac{\zeta(2(w+1))\zeta\left(u+\frac{w+1}{2}\right)\zeta\left(u-\frac{w+1}{2}\right)}{\zeta(w+2)}\int_{-1}^{1}(1-x^2)^{\frac w2}\theta_{\infty}\left(x\right)dx\right]dwdu
\end{align*}
As in Proposition \ref{propdom1} we can move the $u$-contour in the second integral to $\Re(u)=-\frac12$ picking only the residue of $\frac1u$, which implies that the second integral is $O_{k}(1)$. 

For the first line, we again move the $u$-contour to $\Re(u)=-\frac12$ (Note, once again, that the $x$-integral still converges since we are assuming $k\geq3$ which implies $\theta_{\infty}(x)$ vanishes to order at least $2$ at $x=\pm1,$ hence $(1-x^2)^{-3/2}\theta_{\infty}(x)$ is integrable around $x=\pm1$.). This picks up the pole of $\zeta(4u-4)$ at $u=\frac54$ with residue $\frac14$, the pole of $\tilde{F}(2u-2)$ at $u=1$ with residue $\frac12$, and the pole of $\frac1u$ at $s=0$ with residue $1$. Therefore, 
\begin{align*}
\eqref{prop2.9*}&=2^{\frac{-1}{2}}X^{\frac54}\tilde{F}\left(\tfrac12\right)\int \tfrac{\theta_{\infty}(x)}{(1-x^2)^{\frac14}}dx+\pi X\zeta(0)\int\tfrac{\theta_{\infty}(x)}{\sqrt{1-x^2}}dx+O_{k}(1).
\end{align*}
The lemma follows from substituting $\zeta(0)=-1/2$.

\end{proof}

\begin{cor}\label{cordom} For $k\geq3$ we have
\begin{align*}
\sum_{n<X}\sqrt{n}\sum_{f,l=1}^{\infty}\tfrac{1}{(lf^2)^{\frac32}}\tfrac{Kl_{l,f}(0,n)}{\sqrt{l}}I_{l,f}(\xi,n)&=\tfrac{X}{k-1}+O_{k}(1).
\end{align*}
\end{cor}

\begin{proof}By propositions \ref{propdom1} and \ref{propdom2} the LHS is
\begin{multline*}
2X^{\frac32}\zeta(2)\int\theta_{\infty}(x)dx+2^{\frac{-1}{2}}X^{\frac54}\tilde{F}\left(\tfrac{-1}{2}\right)\int \tfrac{\theta_{\infty}(x)}{(1-x^2)^{\frac14}}dx\\
+2^{\frac{-1}{2}}X^{\frac54}\tilde{F}\left(\tfrac12\right)\int \tfrac{\theta_{\infty}(x)}{(1-x^2)^{\frac14}}dx+\pi \zeta(0)X\int\tfrac{\theta_{\infty}(x)}{\sqrt{1-x^2}}dx+O_{k}(1).
\end{multline*}
By lemma 3.3 of  \cite{Altug:2015aa} $\tilde{F}$ is an odd function hence the second and third terms cancel each other (We remark that this is not a consequence of our particular choice of the function $F$. If, instead of the current choicee, we had chosen an $F$ so that $\tilde{F}$ is not odd, the function $H$ that appear in the approximate functional equation would change and the third term above would have $-\tilde{F}\left(\tfrac{-1}{2}\right)$ instead of $\tilde{F}\left(\tfrac{1}{2}\right)$.).

As we already remarked in proposition \ref{propdom1} the first term is the contribution of the trivial representation. In general we would have to remove this term since we are only interested in the cuspidal part of the spectrum. In the current case, the operator used in the Selberg trace formula is a projection on a subset of the custpidal spectrum and since the trivial representation is orthogonal to the trivial representation this contribution is $0$. We will show this fact directly by showing that the integral vanishes. 

By the definition of the function $\theta_{\infty}(x)$ given in lemma \ref{ellem1}
\begin{align*}
\int \theta_{\infty}(x)dx&=\tfrac{i}{2\pi}\int_{-1}^1\left\{(x+\sqrt{x^2-1})^{k-1}-(x-\sqrt{x^2-1})^{k-1}\right\}dx\\
&=\tfrac{i^k}{2\pi}\int_{-\frac\pi2}^{\frac{\pi}{2}}(e^{-i(k-1)\alpha}-(-1)^{k-1}e^{i(k-1)\alpha})\cos(\alpha)d\alpha.
\end{align*}
Since $k\equiv 0\bmod2,$ $(-1)^{k-1}=-1$ and we have
\begin{align*}
\tfrac{i^k}{2\pi}\int_{-\frac\pi2}^{\frac{\pi}{2}}(e^{-i(k-1)\alpha}-(-1)^{k-1}e^{i(k-1)\alpha})\cos(\alpha)d\alpha&=\tfrac{i^k}{\pi}\int_{-\frac\pi2}^{\frac{\pi}{2}}\cos((k-1)\alpha)\cos(\alpha)d\alpha\\
&=\tfrac{i^k}{4\pi}\left.\left(\tfrac{\sin(k\alpha)}{k}+\tfrac{\sin((k-2)\alpha)}{k-2}\right)\right|_{-\frac\pi2}^{\frac\pi2}\\
&=0,
\end{align*}
where in the last line we used $k\equiv0\bmod2$ again. Finally we will calculate the integral that appears in the fourth term. Proceeding as above, 
\begin{align*}
\int \tfrac{\theta_{\infty}(x)}{\sqrt{1-x^2}}dx&=\tfrac{i}{2\pi}\int_{-1}^1\tfrac{(x+\sqrt{x^2-1})^{k-1}-(x-\sqrt{x^2-1})^{k-1}}{\sqrt{1-x^2}}dx\\
&=\tfrac{i^k}{2\pi}\int_{-\frac\pi2}^{\frac{\pi}{2}}(e^{-i(k-1)\alpha}-(-1)^{k-1}e^{i(k-1)\alpha})d\alpha\\
&=\tfrac{i^k}{2\pi}\int_{-\frac\pi2}^{\frac{\pi}{2}}\cos((k-1)\alpha)d\alpha\\
&=\tfrac{2}{\pi(1-k)},
\end{align*}
where in the last line we used $k\equiv0\bmod 2$. The corollary now follows from the fact that $\zeta(0)=-1/2$.

\end{proof}

\end{section}

\subsection{Analysis of the terms $\xi\neq0$}\label{secell2}

\subsubsection{Preliminary estimates}\label{secell21}

In this section we give estimates on \eqref{afterpois} for $\tfrac{lf^2\xi}{\sqrt{n}}$ running on certain ranges. We emphasize that these estimates are valid for every $n$. In other words, we are not yet bringing the $n$-sum of \eqref{be} in. 

For any integer $n\geq1$ and $\kappa,\alpha>0$ let,
\begin{align}
S_0(n,\kappa)&:=\sum_{\substack{ lf^2\xi \gg n^{\frac12+\kappa}}}\tfrac{Kl_{l,f}(\xi,n)}{\sqrt{l}}\tfrac{I_{l,f}(\xi,n)}{(lf^2)^{\frac32}},\\
S_1(n,\kappa,\alpha)&:=\sum_{\substack{lf^2\xi \ll n^{\frac12+\kappa}\\ \xi\gg n^{\frac{1}{6}+\kappa+\alpha}}}\tfrac{Kl_{l,f}(\xi,n)}{\sqrt{l}}\tfrac{I_{l,f}(\xi,n)}{(lf^2)^{\frac32}},\\
S
_2(n,\kappa,\alpha)&:=\sum_{\substack{ lf^2\ll n^{\frac14-\kappa} \\ 0<\xi\ll n^{\frac16+\kappa+\alpha}}}\tfrac{Kl_{l,f}(\xi,n)}{\sqrt{l}}\tfrac{I_{l,f}(\xi,n)}{(lf^2)^{\frac32}}.
\end{align}

\paragraph{A heuristic discussion.}To make the proofs easier to follow let us begin by giving a heuristic discussion of theorems \ref{sec3prop1}, \ref{sec3prop3}, and \ref{sec3prop4}. There are certain structural points about these sums that we first would like to discuss, which clarifies the estimates a bit.  

\begin{enumerate}

\item First of all, note that in all of the estimates below the parameters $l$ and $f$ appear as the product $lf^2$ rather than individually. So it is the range for which $lf^2$ runs over that matters. 

\item Next, by\footnote{There is an extra $log(lf^2)$ factor in Corollary B.8 of \cite{Altug:2015ab} however for square-free $l$ one can remove that factor. This, anyways, is not the main point of the heuristic discussion.} Corollary B.8 of \cite{Altug:2015ab}, for a square-free $l$, and $\xi$ and $n$ such that $\gcd(f,n)=1$ and $\gcd(l,\xi)=1$, we have $Kl_{l,f}(\xi,n)\ll \sqrt{l}$. Therefore the quotient $\frac{Kl_{l,f}(\xi,n)}{\sqrt{l}}$, at least when the above conditions are satisfied, is a complex number of absolute value $1$, and for the heuristics we may as well assume that this is the case disregarding the conditions. 

\item Finally, by Proposition 4.7 \cite{Altug:2015ab} (where their $\theta_{\infty,1}^{pos}$ is our $\theta_{\infty}$) we have $I_{l,f}(\xi,n)\ll \left(\tfrac{\sqrt{n}}{lf^2\xi}\right)^N\frac{1}{\xi^2}$ (cf. \eqref{prop3.1eq2} below) for every $N\geq0$ and every $l,f,n,\xi\neq0$, and by Proposition 4.9 of loc. cit. we have $I_{l,f}(\xi,n)\ll\left(\tfrac{\xi\sqrt{n}}{lf^2}\right)^{-\frac32}$ for $\frac{lf^2\xi}{\sqrt{n}}\ll 1$ (cf. \eqref{rangeest} below). 

\end{enumerate}

We can now discuss the contents of the theorems below. By the third remark above, $I_{l,f}(\xi,n)$ decays very rapidly when $lf^2\xi\gg \sqrt{n}$. Therefore we expect $S_0$ to be very small and this is the content of Theorem \ref{sec3prop1}.

By the same remark, when $lf^2\xi\ll \sqrt{n}$ we have $\frac{I_{l,f}(\xi,n)}{(lf^2)^{\frac32}}\ll (\sqrt{n}\xi)^{-\frac32}$. Therefore, $S_1$ is going to be comparable to the double integral
\[n^{-\frac34}\int_{n^{\frac16}}^{n^{\frac12}}\xi^{-\frac32}\int_1^{\frac{n^{\frac12}}{\xi}}1=O(n^{-\frac12}).\]
This is the content of Theorem \ref{sec3prop3}.

 When $lf^2\ll n^{\frac14}$ and $\xi\ll n^{\frac16}$ we necessarily have $lf^2\xi\ll n^{\frac5{12}}<\sqrt{n}$. Hence, again by the same remark, in this region $\frac{I_{l,f}(\xi,n)}{(lf^2)^{\frac32}}\ll (\sqrt{n}\xi)^{-\frac32}$. Therefore, we expect $S_2$ to be $O(n^{-\frac12})$ and this is the content of Theorem \ref{sec3prop4}.

\vspace{0.2in}

\begin{thm}\label{sec3prop1} Let $n\in \mathbb{Z}_{>0}$. Then for every $N>0$ and $\kappa>0$,
\[S_0(n,\kappa)\ll n^{\frac12-N\kappa},\]
where the implied constant depends only on $k,N,$ and $\kappa$.
\end{thm}

\begin{proof} 
 
We recall two estimates from \cite{Altug:2015ab} that will be used throughout the proof. First, corollary B.8 of \cite{Altug:2015ab} states,
\begin{equation}\label{prop3.1eq1}
Kl_{l,f}(\xi,n)\ll \delta(n;f^2)\log(lf^2)\sqrt{l\gcd(n,f^2)}\sqrt{\gcd\left(\tfrac{\xi}{\sqrt{\gcd(n,f^2)}},l\right)},
\end{equation}
where $\delta(n;f^2)$ is $1$ if $n$ is a square modulo $f^2$ and $0$ otherwise. The second estimate, corollary 4.8 of \cite{Altug:2015ab}, is on the Fourier transform that appear in the sum. It states that for every $N>0$ 
\begin{equation}\label{prop3.1eq2}
I_{l,f}(\xi,n)\ll \left(\tfrac{\sqrt{n}}{lf^2\xi}\right)^{N}\tfrac{1}{\xi^2},
\end{equation}
where the implied constant depends only on $k$ and $N$. Using these we can now prove the proposition. First, we will show that the triple sum of the proposition converges absolutely. Note that 
\[lf^2\xi\gg n^{\frac{1}{2}+\kappa}\hspace{0.5in}\Rightarrow \hspace{0.5in}\tfrac{\sqrt{n}}{lf^2\xi}\ll n^{-\kappa}.\]
Combining this bound and \eqref{prop3.1eq2} we get that for every $N>0,$
\begin{align*}
S_0(n,\kappa)&\ll n^{-N\kappa}\sum_{f,l=1}^{\infty}\tfrac{1}{(lf^2)^{\frac32}}\sum_{\substack{ lf^2\xi \gg n^{\frac12+\kappa}}}\tfrac{Kl_{l,f}(\xi,n)}{\sqrt{l}}\tfrac{1}{\xi^2}\\
&\ll n^{-N\kappa}\sum_{f,l=1}^{\infty}\tfrac{1}{(lf^2)^{\frac32}}\sum_{\substack{ \xi \in \mathbb{Z}\backslash \{0\}}}\tfrac{Kl_{l,f}(\xi,n)}{\sqrt{l}}\tfrac{1}{\xi^2},\tag{$\circ$}\label{prop3.1eq3}
\end{align*}
where the implied constant depends only on $k$ and $N$.
Next, we trivially have
\[\gcd(n,f^2)\leq n\hspace{0.5in}and\hspace{0.5in}\sqrt{\gcd\left(\tfrac{\xi}{\sqrt{\gcd(n,f^2)}},l\right)}\leq \sqrt\xi.\]
Substituting these two bounds in \eqref{prop3.1eq1} gives
\begin{equation}\label{prop3.1eq4}
Kl_{l,f}(\xi,n)\ll \log(lf^2)\sqrt{nl\xi}.
\end{equation}
Using \eqref{prop3.1eq4} in \eqref{prop3.1eq3} shows that 
\begin{equation*}
S_0(n,\kappa)\ll n^{\frac12-N\kappa}\sum_{\xi,f,l=1}^{\infty}\tfrac{\log(lf^2)}{(lf^2\xi)^{\frac32}}\ll n^{\frac12-N\kappa}.
\end{equation*}

\end{proof}

\begin{cor}\label{cors0}For every $N,\kappa>0$, 
\begin{equation*}
\sum_{n<X}\sqrt{n}\, S_0(n,\kappa)=O(X^{2-N\kappa}),
\end{equation*}
where the implied constant depends only on $k,\kappa,$ and $N$.
\end{cor}

\begin{proof}
Follows from Theorem \ref{sec3prop1}.
\end{proof}


\begin{thm}\label{sec3prop3}Let $n\in \mathbb{Z}_{>0}$. Then for every $\kappa,\alpha>0$,
\[S_1(n,\kappa,\alpha)\ll \tfrac{\log(n)}{n^{\frac{3-2\kappa}{4}}}\Bigl(\sum_{\substack{d_0^2\mid n\\ d_0\ll n^{\frac{1+2\kappa}6}}}\tfrac{\log(n)}{d_0^{1/2}}+\tfrac{n^{\frac{1-2(\kappa+3\alpha)}{4}}}{d_0^{3/2}}\Bigr),\]
where the implied constant depends only on $k,\kappa,$ and $\alpha$.

\end{thm}

\begin{proof}To prove the theorem we will bound
\[S_1(n,\kappa,\delta):=\sum_{\substack{lf^2\xi \ll n^{\frac12+\kappa}\\ \xi\gg n^{\delta}}}\tfrac{Kl_{l,f}(\xi,n)}{l^2f^3}I_{l,f}(\xi,n),\]
for any $\delta>\kappa$ (not to be confused with the function $\delta(n;f^2)$ of \eqref{rangecharbound}) and then specialize to $\delta=\frac16+\kappa+\alpha$. This on one hand avoids notational burden, and on the other hand proves that the exponent $\frac16$ is the best possible via this argument. Let us first estimate the integral. We claim that
\begin{align*}
\int_{-1}^1\theta_{\infty}(x)F\left(\tfrac{lf^2(4n)^{-1/2}}{\sqrt{1-x^2}}\right)e\left(\tfrac{-x\xi \sqrt{n}}{2lf^2}\right)dx&\ll\left(\tfrac{lf^2}{\sqrt{n}}\right)^{\frac32}\tfrac{1}{\xi^{3/2}}\tag{$\star$}\label{rangestar},\\
\int_{-1}^1\tfrac{\theta_{\infty}(x)}{\sqrt{1-x^2}}H\left(\tfrac{lf^2(4n)^{-1/2}}{\sqrt{1-x^2}}\right)e\left(\tfrac{-x\xi \sqrt{n}}{2lf^2}\right)dx&\ll\tfrac{lf^2}{\sqrt{n}\xi}\tag{$\star\star$}\label{rangestarstar},
\end{align*}
where the implied constants depend only on $k$. These bounds, in fact, are the same bounds in proposition 4.9, lines (i) and (iii), of \cite{Altug:2015ab}. The only difference is that in loc. cit. the bounds are proved for the range $\tfrac{lf^2\xi}{\sqrt{n}}\ll 1$, however our range is $\frac{lf^2\xi}{\sqrt{n}}\ll n^{\kappa}$ and $\xi\gg n^{\frac16+2\kappa}$. We claim that the bounds still hold for this range. This indeed follows from the proof of proposition 4.9 of loc.cit.. In order to give the details we first need to remind the reader how the assumption $\tfrac{lf^2\xi}{\sqrt{n}}\ll1$ was used in the proof of proposition 4.9. First, the proof itself depends on the asymptotic expansion of theorem A.14 of loc.cit.. The assumption $\tfrac{lf^2\xi}{\sqrt{n}}\ll1$ implies that $\tfrac{\xi\sqrt{n}}{lf^2}=\tfrac{\sqrt{n}}{lf^2\xi}\xi^2\gg1$ ,and substituting this for the parameter $D$ in the asymptotic expansion of theorem A.14 implies that the first term dominates the asymptotic expansion, which is the result of proposition 4.9. In other words, the essential ingredient is the bound $\frac{\sqrt{n}\xi}{lf^2}\gg 1$.

Coming back to our proof we have $\xi\gg n^{\delta}$, which implies that $\tfrac{\xi\sqrt{n}}{lf^2}\gg \tfrac{n^{\frac12+\delta}}{lf^2}\gg n^{\delta-\kappa}\gg1$, hence the conclusion of proposition 4.9 of \cite{Altug:2015ab} still holds. Note also that
\[lf^2\xi\ll n^{\frac12+\kappa}\hspace{0.5in}\Rightarrow\hspace{0.5in}\tfrac{lf^2}{\sqrt{n}}\ll \tfrac{n^{\kappa}}{\xi}\tag{$\star\star\star_{S_1}$}\label{rangestarstarstar}.\]
Substituting \eqref{rangestarstarstar} in \eqref{rangestarstar} then gives, 
\[\int_{-1}^1\tfrac{\theta_{\infty}(x)}{\sqrt{1-x^2}}H\left(\tfrac{lf^2(4n)^{-1/2}}{\sqrt{1-x^2}}\right)e\left(\tfrac{-x\xi \sqrt{n}}{2lf^2}\right)dx\ll\left(\tfrac{lf^2}{\sqrt{n}}\right)^{\frac12}\tfrac{n^{\frac\kappa2}}{\xi^{3/2}}\tag{$\star\star_{S_1}$}\label{rangestarstarprime}.\]
Using \eqref{rangestar} and \eqref{rangestarstarprime} we then get
\begin{equation}\label{rangeest}
I_{l,f}(\xi,n)\ll \left(\tfrac{lf^2}{\sqrt{n}}\right)^{\frac32}\tfrac{n^{\frac\kappa2}}{\xi^{3/2}},
\end{equation}
where the implied constants depend only $k$ and $\kappa$. Next, we bound the character sums, $Kl_{l,f}(\xi,n)$. By corollary B.8 of \cite{Altug:2015ab} we have,
\begin{equation}\label{rangecharbound}
Kl_{l,f}(\xi,n)\ll \begin{cases}\delta(n;f^2)\log(lf^2)\sqrt{l\gcd(n,f^2)}\sqrt{\gcd\left(\tfrac{\xi}{\sqrt{\gcd(n,f^2)}},l\right)}& \frac{l\sqrt{\gcd(n,f^2)}}{rad(l)}\mid \xi\\ 0& otherwise\end{cases},
\end{equation}
where $\delta(n;f^2)$ is $1$ if $n$ is a square modulo $f^2$ and $0$ otherwise, and $rad(l)=\prod_{p\mid l}p$, where $p$ denotes a prime. This, imn particular, implies that if $\gcd(n,f^2)$ is not a perfect square then $Kl_{l,f}(\xi,n)$ vanishes. For $l,f,\xi,n$ such that $Kl_{l,f}(\xi,n)\neq0$ let,
\begin{align*}
d_0^2:=\gcd(n,f^2),\hspace{0.5in} &n=d_0^2n_0,\hspace{0.5in}f=d_0f_0,\hspace{0.5in}\xi=d_0\xi_0,\\
&\hspace{-0.7in}d_1:=\gcd(l,\xi_0),\hspace{0.5in}l=d_1l_0,\hspace{0.5in}\xi_0=d_1\xi_1.
\end{align*}
Note also that $l_0$ is square-free. Substituting these in \eqref{rangecharbound} gives
\begin{equation}\label{rangecharbound2}
Kl_{l,f}(\xi,n)\ll \log(d_0^2d_1l_0f_0^2)d_0d_1\sqrt{l_0}.
\end{equation}
Combining \eqref{rangeest} and \eqref{rangecharbound2} implies that,
\[S_1(n,\kappa,\delta)\ll \tfrac{1}{n^{\frac{3-2\kappa}4}}\sum_{d_0^2\mid n}\sum_{\substack{ d_0^3d_1^2l_0f_0^2\xi_1\ll n^{\frac{1}{2}+\kappa}\\ d_0d_1\xi_1\gg n^{\delta}}}\tfrac{\log(d_0^2d_1l_0f_0^2)}{\sqrt{d_0}d_1\xi_1^{3/2}}.\]
Note that $d_0^3d_1^2l_0f_0^2\xi_1\ll n^{\frac12+\kappa}\Rightarrow d_0\ll n^{\frac{1+2\kappa}6}$. Therefore, in the summation range for $S_1(n,\kappa)$, $d_0$ runs through square divisors of $n$ which are $\ll n^{\frac{1+2\kappa}{6}}$. Therefore,
\begin{equation}\label{rangesplit}
S_1(n,\kappa,\delta)\ll  \tfrac{1}{n^{\frac{3-2\kappa}4}}\Bigl(\sum_{\substack{d_0^2\mid n\\ d_0\ll n^{\frac{1+2\kappa}6}}}\tfrac{T(d_0,n,\kappa,\delta)}{\sqrt{d_0}}+ \sum_{\substack{d_0^2\mid n\\ d_0\ll n^{\frac{1+2\kappa}6}}}\tfrac{U(d_0,n,\kappa,\delta)}{\sqrt{d_0}}\Bigr),
\end{equation}
where 
\[T(d_0,n,\kappa,\delta):=\sum_{\substack{d_0^3d_1^2l_0f_0^2\xi_1\ll n^{\frac{1}{2}+\kappa}\\ d_0d_1\gg n^{\delta}}}\tfrac{\log(d_0^2d_1l_0f_0^2)}{d_1\xi_1^{3/2}},\hspace{0.15in} U(d_0,n,\kappa,\delta):=\sum_{\substack{d_0^3d_1^2l_0f_0^2\xi_1\ll n^{\frac{1}{2}+\kappa}\\  d_0d_1\ll n^{\delta}\\ d_0d_1\xi_1\gg n^{\delta}}}\tfrac{\log(d_0^2d_1l_0f_0^2)}{d_1\xi_1^{3/2}}.\]

We now bound $T(d_0,n,\kappa,\delta)$ and $U(d_0,n,\kappa,\delta)$.
\begin{align*}
T(d_0,n,\kappa,\delta)&= \sum_{\substack{d_0^3d_1^2\ll n^{\frac12+\kappa}\\ d_0d_1\gg n^{\delta}}}\tfrac{1}{d_1}\sum_{f_0^2\ll \lfloor\frac{n^{\frac12+\kappa}}{d_0^3d_1^2}\rfloor}\sum_{l_0\ll \lfloor\frac{n^{\frac12+\kappa}}{d_0^3d_1^2f_0^2}\rfloor}\log(d_0^2d_1l_0f_0^2)\sum_{\xi_1\ll \lfloor\frac{n^{\frac12+\kappa}}{d_0^3d_1^2f_0^2l_0}\rfloor }\tfrac{1}{\xi_1^{3/2}}\\
&\ll \sum_{\substack{d_0^3d_1^2\ll n^{\frac12+\kappa}\\ d_0d_1\gg n^{\delta}}}\tfrac{1}{d_1}\sum_{f_0^2\ll \lfloor\frac{n^{\frac12+\kappa}}{d_0^3d_1^2}\rfloor}\sum_{l_0\ll \lfloor\frac{n^{\frac12+\kappa}}{d_0^3d_1^2f_0^2}\rfloor}\log(d_0^2d_1l_0f_0^2)\\
&\ll \sum_{\substack{d_0^3d_1^2\ll n^{\frac12+\kappa}\\ d_0d_1\gg n^{\delta}}}\tfrac{1}{d_1}\sum_{f_0^2\ll \lfloor\frac{n^{\frac12+\kappa}}{d_0^3d_1^2}\rfloor}\left(\log(d_0^2d_1f_0^2)+\left(1+\tfrac{n^{1/2+\kappa}}{d_0^3d_1^2f_0^2}\right)\log\left(\tfrac{n^{1/2+\kappa}}{d_0d_1}\right)\right)\\
&\ll \sum_{\substack{d_0^3d_1^2\ll n^{\frac12+\kappa}\\ d_0d_1\gg n^{\delta}}}\left(\tfrac{\log(d_0^2d_1)}{d_1}+\tfrac{n^{1/2+\kappa}}{d_0^3d_1^3}\log\left(\tfrac{n^{1/2+\kappa}}{d_0d_1}\right)\right)\\
&\ll \log^2\left(n\right)+\tfrac{n^{1/2+\kappa-2\delta}}{d_0}\log(n).\tag{$\circ$}\label{rangebigbound1}
\end{align*}
We remark that all of the implied constants depend only on $\kappa$ and $\delta$. Moving on to $U(d_0,n,\kappa,\delta)$ we have:
\begin{align*}
U(d_0,n,\kappa,\delta)&= \sum_{\substack{d_0^3d_1^2\ll n^{\frac12+\kappa}\\ d_0d_1\ll n^{\delta}}}\tfrac{1}{d_1}\sum_{\substack{\lfloor\frac{n^{\frac12+\kappa}}{d_0^3d_1^2}\rfloor\gg \xi_1\gg \lfloor\frac{n^{\delta}}{d_0d_1}\rfloor}}\tfrac{1}{\xi_1^{3/2}}\sum_{f_0^2\ll \lfloor\frac{n^{\frac12+\kappa}}{d_0^3d_1^2\xi_1}\rfloor}\sum_{l_0\ll \lfloor\frac{n^{\frac12+\kappa}}{d_0^3d_1^2f_0^2\xi_1}\rfloor}\log(d_0^2d_1l_0f_0^2)\\
&\ll \sum_{\substack{d_0^3d_1^2\ll n^{\frac12+\kappa}\\ d_0d_1\ll n^{\delta}}}\tfrac{1}{d_1}\sum_{\substack{\lfloor\frac{n^{\frac12+\kappa}}{d_0^3d_1^2}\rfloor\gg \xi_1\gg\lfloor\frac{n^{\delta}}{d_0d_1}\rfloor}}\tfrac{1}{\xi_1^{3/2}}\sum_{f_0^2\ll \lfloor\frac{n^{\frac12+\kappa}}{d_0^3d_1^2\xi_1}\rfloor}\left(\log\left({d_0^2d_1f_0^2}\right)+\left(1+\tfrac{n^{1/2+\kappa}}{d_0^3d_1^2f_0^2\xi_1}\right)\log\left(\tfrac{n^{1/2+\kappa}}{d_0d_1\xi_1}\right)\right)\\
&\ll \sum_{\substack{d_0^3d_1^2\ll n^{\frac12+\kappa}\\ d_0d_1\ll n^{\delta}}}\tfrac{1}{d_1}\sum_{\substack{ \xi_1\gg \lfloor\frac{n^{\delta}}{d_0d_1}\rfloor}}\left(\tfrac{\log(d_0^2d_1)}{\xi^{3/2}}+\left(\tfrac{n^{1/2+\kappa}}{d_0^3d_1^2\xi_1^{5/2}}+\tfrac{1}{\xi^{3/2}}\right)\log\left(\tfrac{n^{1/2+\kappa}}{d_0d_1\xi_1}\right)\right)\\
&\ll \sum_{\substack{d_0^3d_1^2\ll n^{\frac12+\kappa}\\ d_0d_1\ll n^{\delta}}}\left(\tfrac{\sqrt{d_0}\log(d_0^2d_1)}{n^{\delta/2}\sqrt{d_1}}+\tfrac{n^{(1+2\kappa-3\delta)/2}\log(n)}{\sqrt{d_0^3d_1^3}}\right)\\
&\ll \Bigl(1+\tfrac{\sqrt{d_0}}{n^{\delta/2}}+\tfrac{n^{(1+2\kappa-3\delta)/2}}{d_0^{3/2}}\Bigr)\log(n).\tag{$\circ\circ$}\label{rangebigbound2}
\end{align*}
Once again, all of the implied constants depend only on $\kappa$ and $\delta$.

Finally, substituting \eqref{rangebigbound1} and \eqref{rangebigbound2} into \eqref{rangesplit} gives,
\begin{align*}
S_1(n,\kappa,\delta)\ll \tfrac{\log(n)}{n^{\frac{3-2\kappa}{4}}}\Bigl(\sum_{\substack{d_0^2\mid n\\ d_0\ll n^{\frac{1+2\kappa}6}}}\tfrac{\log(n)}{d_0^{1/2}}+\tfrac{n^{(1+2\kappa-3\delta)/2}}{d_0^{3/2}}+\tfrac{1}{n^{\delta/2}}\Bigr),
\end{align*}
where the implied constant depends only on $k,\kappa,$ and $\delta$. The theorem then follows from substituting $\delta=\frac16+\kappa+\alpha$.

\end{proof}

\begin{cor}\label{cors1}For every $\kappa,\alpha>0$,
\begin{equation*}
\sum_{n<X}\sqrt{n}\,S_1(n,\kappa,\alpha)\ll \Bigl(\log^2(X)X^{\frac{3+2\kappa}{4}}+\log(X)X^{1-\frac{3\alpha}{2}}\Bigr).
\end{equation*}
where the implied constant depends only on $k,\kappa,$ and $\alpha$.
\end{cor}

\begin{proof}
By theorem \ref{sec3prop3} we get
\begin{align*}
\sum_{n<X}\sqrt{n}\,S_1(n,\kappa,\alpha)&\ll\sum_{n<X}\tfrac{\log(n)}{n^{\frac{1-2\kappa}{4}}}\Bigl(\sum_{\substack{d_0^2\mid n\\ d_0\ll n^{\frac{1+2\kappa}6}}}\tfrac{\log(n)}{d_0^{1/2}}+\tfrac{n^{\frac{1-2(\kappa+3\alpha)}{4}}}{d_0^{3/2}}\Bigr)\\
&\ll \sum_{n< X}\tfrac{\log(n)}{n^{\frac{1-2\kappa}{4}}}\Bigl(\sum_{\substack{d_0^2\mid n\\ d_0\ll n^{\frac{1+2\kappa}6}}}\tfrac{\log(n)}{d_0^{1/2}}+\tfrac{n^{\frac{1-2(\kappa+3\alpha)}{4}}}{d_0^{3/2}}\Bigr).
\end{align*}
Bounding each term in the double sum separately we get,
\begin{align*}
\sum_{\substack{n< X}}\tfrac{\log^2(n)}{n^{\frac{1-2\kappa}4}}\sum_{\substack{d_0^2\mid n\\ d_0\ll n^{\frac{1+2\kappa}{6}}}}\tfrac{1}{d_0^{1/2}}&= \sum_{ d_0\ll X^{\frac{1+2\kappa}6} }\tfrac{1}{d_0^{1-\kappa}}\sum_{n\ll \lfloor \frac{X}{d_0^2}\rfloor}\tfrac{\log^2(nd_0^2)}{n^{\frac{1-2\kappa}{4}}}\\
&\ll  \sum_{ d_0\ll X^{\frac{1+2\kappa}6} }\tfrac{\log^2(X)}{d_0^{1-\kappa}}\Bigl(\Bigl(\tfrac{X}{d_0^2}\Bigr)^{\frac{3+2\kappa}{4}}+\Bigl(\tfrac{d_0^2}{X}\Bigr)^{\frac{1-2\kappa}{4}}+1\Bigr)\ll \log^2(X)X^{\frac{3+2\kappa}{4}}, \\
\sum_{\substack{n< X}}\tfrac{\log(n)}{n^{3\alpha/2}}\sum_{\substack{d_0^2\mid n\\ d_0\ll n^{\frac{1+2\kappa}{6}}}}\tfrac{1}{d_0^{3/2}}&=\sum_{d_0\ll X^{\frac{1+2\kappa}{6}}}\tfrac{1}{d_0^{(3+6\alpha)/2}}\sum_{n\ll \lfloor\frac{X}{d_0^2}\rfloor}\tfrac{\log(nd_0^2)}{n^{3\alpha/2}}\\
&\ll \sum_{d_0\ll X^{\frac{1+2\kappa}{6}}}\tfrac{\log(X)}{d_0^{(3+6\alpha)/2}}\Bigl(\left(\tfrac{X}{d_0^2}\right)^{1-\frac{3\alpha}{2}}+\left(\tfrac{d_0^2}{X}\right)^{\frac{3\alpha}{2}}+1\Bigr)\ll \log(X)X^{1-\frac{3\alpha}{2}}.
\end{align*}
The corollary follows.

\end{proof}

\begin{thm}\label{sec3prop4}Let $n\in\mathbb{Z}_{>0}$. Then for every $\kappa>0$ and $\frac{1}{12}\geq\alpha>0$,

\[S_2(n,\kappa,\alpha)\ll \log^2(n)n^{-\frac{1}2-\kappa},\]
where the implied constant depends only on $k,\kappa,$ and $\alpha$.

\end{thm}

\begin{proof} The proof follows the same argument as in the proof of Theorem \ref{sec3prop3}. First, note that since $lf^2\ll n^{\frac14-\kappa}$ and $\xi\neq0$ we necessarily have $\frac{\xi\sqrt{n}}{lf^2}\gg 1$. Therefore the first paragraph of the proof of Theorem \ref{sec3prop3} goes through verbatim and implies the same bounds as in \eqref{rangestar} and \eqref{rangestarstar} of that theorem. Moreover, since $S_2(n,\kappa,\alpha)$ has the summation ranges $lf^2\ll n^{\frac14-\kappa}$ and $\xi\ll n^{\frac16+\kappa+\alpha}$ we have $lf^2\xi\ll n^{\frac{5}{12}+\alpha}$. Therefore,
\begin{equation*}
\tfrac{lf^2}{\sqrt{n}}\ll \tfrac{1}{n^{1/12-\alpha}\xi}\tag{$\star\star\star_{S_2}$}\label{3stars3}
\end{equation*} 
Substituting \eqref{3stars3} into \eqref{rangestarstar} and using $\alpha\leq\frac1{12}$ gives
\[\int_{-1}^1\tfrac{\theta_{\infty}(x)}{\sqrt{1-x^2}}H\left(\tfrac{lf^2(4n)^{-1/2}}{\sqrt{1-x^2}}\right)e\left(\tfrac{-x\xi \sqrt{n}}{2lf^2}\right)dx\ll\left(\tfrac{lf^2}{\sqrt{n}}\right)^{\frac12}\tfrac{1}{\xi^{3/2}}.\tag{$\star\star_{S_2}$}\label{2stars3}\]
Finally, substituting \eqref{rangestar} and \eqref{2stars3} in $I_{l,f}(\xi,n)$ we get
\begin{equation}\label{sec3thm3eq1}
I_{l,f}(\xi,n)\ll \left(\tfrac{lf^2}{\sqrt{n}}\right)^{\frac32}\tfrac{1}{\xi^{3/2}}.
\end{equation}
We also remark that the implied constant is independent of $l,f,\xi$ and $n$. Moving on to the character sum, $Kl_{l,f}(\xi,n)$, we once again have the bound in \eqref{rangecharbound}. Note that because of the presence of $\delta(n;f^2)$ implies that the $\gcd(n,f^2)$ has to be a perfect square otherwerwise the sum vanishes. Moreover, when $\gcd(n,f^2)$ is a square, the sum still vanishes unless $\sqrt{\gcd(n,f^2)}\mid \xi$. This implies that whenever $Kl_{l,f}(\xi,n)\neq0$ we have to have $\xi=\sqrt{\gcd(n,f^2)}\xi_1$ for some $\xi_1$, and we have the bound
\begin{equation}\label{rangecharbound12}
Kl_{l,f}(\xi,n)=Kl_{l,f}(\gcd(n,f^2)\xi_1,n)\ll \log(lf^2)\sqrt{l\gcd(n,f^2)\xi_1}=log(lf^2)\sqrt{l\xi}.
\end{equation}
Substituting the bounds in \eqref{sec3thm3eq1} and \eqref{rangecharbound12} into $S_{2}(n,\kappa)$ gives,
\begin{align*}
S_2(n,\kappa,\alpha)&\ll \sum_{lf^2\ll n^{\frac{1}{4}-\kappa}}\tfrac{1}{(lf^2)^{\frac32}}\sum_{\xi\ll n^{\frac16+\kappa+\alpha}}\tfrac{\log(lf^2)\sqrt{l\xi}}{\sqrt{l}}\left(\tfrac{lf^2}{\sqrt{n}}\right)^{\frac32}\tfrac{1}{\xi^{\frac32}}\ll \log^2(n)n^{-\frac{1}2-\kappa}.
\end{align*}

\end{proof}

\begin{cor}\label{cors2}
For every $\kappa>0$ and $\frac{1}{12}\geq\alpha>0$,
\begin{equation*}
\sum_{n<X}\sqrt{n}\,S_2(n,\kappa,\alpha)\ll \log^2(X)X^{1-\kappa},
\end{equation*}
where the implied constant depends only on $k, \kappa$ and $\alpha$.
\end{cor}

\begin{proof}Follows from Theorem \ref{sec3prop4}.
\end{proof}

\subsubsection{Estimating the critical range}\label{secell22}

In this section we will estimate the critical range of summation where $X^{\frac12+\kappa}\gg lf^2 \gg X^{\frac14-\kappa}$. Let
\begin{equation}\label{criticalsum}
S(X,\kappa,\alpha):=\sum_{n<X}\sqrt{n}\sum_{\substack{X^{\frac12+\kappa}\gg lf^2\gg X^{\frac14-\kappa}}}\tfrac{1}{l^2f^3}\sum_{\substack{ \xi\in\mathbb{Z}\backslash \{0\}\\ lf^2\xi\ll X^{\frac12+\kappa}\\ \xi\ll X^{\frac16+\kappa+\alpha}}}Kl_{l,f}(\xi,n)I_{l,f}(\xi,n).
\end{equation}
The basic strategy for estimating this sum is to apply Poisson summation to the $n$-sum. The heuristic reason is pretty simple:

Using a smooth dyadic partition we can assume that $n\sim X$. Moreover, in the region of summation $I_{l,f}(\xi,n)$ is roughly $1$ (we have used all of its decay properties to get the estimates of the previous section). Therefore, bounding \eqref{criticalsum} reduces to bounding
\[\sqrt{X}\sum_{n\in \mathbb{Z}}G\left(\tfrac{n}{X}\right) \sum_{\substack{X^{\frac12}\gg lf^2\gg X^{\frac14}\\ lf^2\xi\ll X^{\frac12}\\ \xi\ll X^{\frac16}}}\frac{Kl_{lf}(\xi,n)}{l^2f^3}.\]
Now notice that $Kl_{l,f}(\xi,n)$ is periodic in $n$ modulo $4lf^2$. Therefore  Poisson summation on the above sum gives
\[X^{\frac32} \sum_{\substack{X^{\frac12}\gg lf^2\gg X^{\frac14}\\ lf^2\xi\ll X^{\frac12}\\ \xi\ll X^{\frac16}}}\frac{1}{l^3f^5}\sum_{n\in \mathbb{Z}}\hat{G}\left(\tfrac{X\nu}{4lf^2}\right)\omega_{lf}(\xi,n),\]
where $\omega_{l,f}(\xi,\nu)$ is as in \eqref{omega}, and it is roughly of size $ lf$ (cf. Corollary \ref{omegacor}). Since $lf^2\ll \sqrt{X}$ we deduce that as long as $\nu\neq0$ the decay of $\hat{G}$ will guarantee that the sum is very small in terms of the variable $X$. The only remaining point is to analyze the term corresponding to $\nu=0$. By a local analysis we can show that this term is $0$ unless $\nu$ and $\xi$ satisfy certain divisibility conditions and in our range of summation those conditions cannot be satisfied (see the proof of Theorem \ref{thmsigma}). 

The only difficulty in executing this simple strategy is that the function $G$ is not alone. It comes as the product of $G(\frac nX)I_{l,f}(\xi,n)$ so that one needs to get the decay properties of the Fourier transform of this product uniformly in all the variables, and this is done in Proposition \ref{archprop} of \S\ref{seclocal}.

We can now go on and execute the strategy described above. First of all, let $G\in C_c^{\infty}([\frac14,\frac54])$ be a smooth function. In Corollary \ref{cors} we will specialize this mollifier to a smooth approximation to the characteristic function of the interval of $[\frac12,1)$ to get back from the estimates on smoothed sums to estimating \eqref{criticalsum}. Set,

\begin{equation}
S_G(X,\kappa,\alpha):=\sum_{n\in \mathbb{Z}}G\left(\tfrac{n}{X}\right)\sqrt{n}\sum_{\substack{X^{\frac12+\kappa}\gg lf^2\gg X^{\frac14-\kappa}}}\tfrac{1}{l^2f^3}\sum_{\substack{ \xi\in\mathbb{Z}\backslash \{0\}\\ lf^2\xi\ll X^{\frac12+\kappa}\\ \xi\ll X^{\frac16+\kappa+\alpha}}}Kl_{l,f}(\xi,n)I_{l,f}(\xi,n).
\end{equation}

\begin{thm}\label{thmsigma}For every $G\in C_c^{\infty}([\frac14,\frac54])$, $M\geq2,$ $\kappa,\alpha>0$ such that $\frac{1}{12}>2\kappa+\alpha$ we have
\[S_G(X,\kappa,\alpha)=O(\|G \|_{{M,1}}X^{\frac{17}{12}+2\kappa+\alpha-M(\frac13-\kappa-\alpha)}\log(X)),\]
where 
\[\|G \|_{M,1}=\sum_{j=0}^M\|G^{(j)}\|_1,\]
$G^{(j)}$ denoting the $j$'th derivative of $G$ (i.e. the Sobolev $W^{M,1}$-norm of $G$), and the implied constant depends only on $k,\kappa,\alpha,$ and $M$.

\end{thm}

\begin{proof}Note that the character sum, $Kl_{l,f}(\xi,n)$, is periodic in $n$ modulo $4lf^2$. Using this we interchange the $n$-sum with the rest of the terms and break it up into arithmetic progressions $\bmod\, 4lf^2$. This gives, 
\begin{equation*}
S_G(X,\kappa,\alpha)=\sum_{X^{\frac12+\kappa}\gg lf^2\gg X^{\frac14-\kappa}}\tfrac{1}{l^2f^3}\sum_{\substack{ \xi\in\mathbb{Z}\backslash \{0\}\\ lf^2\xi\ll X^{\frac12+\kappa}\\ \xi\ll X^{\frac16+\kappa+\alpha}}}\sum_{b\bmod 4lf^2}Kl_{l,f}(\xi,b)\sum_{\substack{n\in\mathbb{Z}\\ n\equiv b\bmod 4lf^2}}\sqrt{n}\,G\left(\tfrac{n}{X}\right)I_{l,f}(\xi,n).
\end{equation*}
We now apply Poisson summation to the $n$-sum and get
\begin{equation}\label{afterpoisson}
S_G(X,\kappa,\alpha)=\sum_{X^{\frac12+\kappa}\gg lf^2\gg X^{\frac14-\kappa}}\tfrac{1}{l^2f^3}\sum_{\substack{ \xi\in\mathbb{Z}\backslash \{0\}\\ lf^2\xi\ll X^{\frac12+\kappa}\\ \xi\ll X^{\frac16+\kappa+\alpha}}}\sum_{\nu\in\mathbb{Z}}\tfrac{\omega_{l,f}(\xi\nu)}{4lf^2}J_{l,f}(\xi,\nu,X),
\end{equation}
where $J_{l,f}(\xi,\nu,X)$ is the Fourier transform defined in \eqref{fourier} and $\omega_{l,f}(\xi,\nu)$ is the character sum defined in \eqref{omega}. In order to bound $J_{l,f,}(\xi,\nu)$ and $\omega_{l,f}(\xi,\nu)$ we will use proposition \ref{archprop} and corollary \ref{omegacor} respectively. For any $M,N\geq0$ and $\nu\neq0$, proposition \ref{archprop} gives the following bound on $I_{l,f}(\xi,\nu)$, 
\[J_{l,f}(\xi,\nu)\ll\| G \|_{M,1}\tfrac{X^{\frac{N-M+ 3}{2}}}{\nu^M(lf^2)^{N}}\left[\left(\tfrac{lf^2}{\sqrt{X}}\right)^M+\xi^M\right],\tag{$\bullet$}\label{lbound1}\]
where the implied constant depends only on $k,G,M,$ and $N$. For $\omega_{l,f}(\xi,\nu)$ corollary \ref{omegacor} gives
\[\omega_{l,f}(\xi,\nu)\ll\begin{cases}\log(lf)lf\sqrt{\gcd(lf^2,\nu)\gcd(l,\nu)}& \frac{l}{rad(l)}\mid \nu,\,\, \gcd(lf^2,\nu)\mid \xi \\ 0& otherwise\end{cases},\tag{$\bullet\bullet$}\label{lbound2}\]
where the implied constant is absolute. Going back to \eqref{afterpoisson} we break the analysis of the $\nu$ sum into two according to $\nu=0$ and $\nu\neq0$.

\begin{itemize}
\item $\nu=0.$ In this case the character sum $\omega_{l,f}(\xi,0)$ does not vanish only if $lf^2\mid \xi$. But the ranges in $S_G(X,\kappa,\alpha)$ are $lf^2\gg X^{\frac14-\kappa}$ and $\xi\ll X^{\frac16+\kappa+\alpha}$. Since $2\kappa+\alpha<\frac{1}{12}$ these ranges don't intersect, hence for all $l,f,\xi$ in the range for $S_G(X,\kappa,\alpha)$ the term corresponding to $\nu=0$ vanishes.

\item $\nu\neq0.$ We will be using the bounds in \eqref{lbound1} and \eqref{lbound2}. We will first need to separate the $\gcd$-factors from \eqref{lbound2}. Let $\gcd(l,\nu)=d_0$ . Then, 
\[\omega_{l,f}(\xi,\nu) \neq0 \hspace{0.3in}\Rightarrow \hspace{0.3in}\begin{setdef}{\substack{l=d_0l_0\\ \nu=d_0\nu_0\\ \xi=d_0\xi_0}}{\substack{\gcd(l_0,\nu_0)=1}}\end{setdef}.\]
Then \eqref{lbound2} implies that,
\[\omega_{d_0l_0,f}(d_0\xi_0,d_0\nu_0)\ll \log(d_0l_0f)d_0^2l_0f\sqrt{\gcd(f^2,\nu_0)}\leq \log(d_0l_0f)d_0^2l_0f\sqrt{\nu_0}.\tag{$\bullet\bullet'$}\label{lbound2'}\]
Now, taking $N=0$ in \eqref{lbound1} and using \eqref{lbound2}, we get the following bound valid for every $M\geq2$, 
\begin{equation*}
S_G^{\nu\neq0}(X,\kappa,\alpha)\ll \|G \|_{{M,1}}X^{\frac{3-M}{2}}\sum_{X^{\frac12+\kappa}\gg d_0l_0f^2\gg X^{\frac14-\kappa}}\tfrac{\log(d_0l_0f)}{d_0(l_0f^2)^2}\sum_{\substack{ \xi_0,\nu_0\in\mathbb{Z}\backslash \{0\}\\ d_0^2l_0f^2\xi_0\ll X^{\frac12+\kappa}\\ d_0\xi_0\ll X^{\frac16+\kappa+\alpha}}}\tfrac{1}{\nu_0^{M-\frac12}}\left[\left(\tfrac{l_0f^2}{\sqrt{X}}\right)^M+\xi_0^M\right],
\end{equation*}
where the implied constant is independent of $X$ and $G$. Since the $\nu_0$-sum converges absolutely (recall  that $M\geq2$) and since $d_0l_0f^2\ll X^{\frac12+\kappa}\Rightarrow \log(d_0l_0f)\leq \log(X)$, we have
\begin{align*}
S_G^{\nu\neq0}(X,\kappa,\alpha)&\ll \|G \|_{{M,1}}X^{\frac{3-M}{2}}\log(X)\sum_{X^{\frac12+\kappa}\gg d_0l_0f^2\gg X^{\frac14-\kappa}}\tfrac{1}{d_0(l_0f^2)^2}\sum_{\substack{ \xi_0\in\mathbb{Z}\backslash \{0\}\\ d_0^2l_0f^2\xi_0\ll X^{\frac12+\kappa} \\d_0\xi_0\ll  X^{\frac16+\kappa+\alpha}}}\left[\left(\tfrac{l_0f^2}{\sqrt{X}}\right)^M+\xi_0^M\right]\\
&=\|G \|_{{M,1}}X^{\frac{3-M}{2}}\log(X)\Bigl(S_{\nu\neq0}^{(1)}(X,\kappa,\alpha)+S_{\nu\neq0}^{(2)}(X,\kappa,\alpha)\Bigr),\tag{$\circ$}\label{fourierendest0}
\end{align*}
where
\begin{align*}
S_G^{\nu\neq0,(1)}(X,\kappa,\alpha):&=X^{-\frac{M}2}\sum_{X^{\frac12+\kappa}\gg d_0l_0f^2\gg X^{\frac14-\kappa}}\tfrac{(l_0f^2)^{M-2}}{d_0}\sum_{\substack{ \xi_0\in\mathbb{Z}\backslash \{0\}\\ d_0^2l_0f^2\xi_0\ll X^{\frac12+\kappa} \\d_0\xi_0\ll  X^{\frac16+\kappa+\alpha}}}1,\\
S_G^{\nu\neq0,(2)}(X,\kappa,\alpha):&=\sum_{X^{\frac12+\kappa}\gg d_0l_0f^2\gg X^{\frac14-\kappa}}\tfrac{1}{d_0(l_0f^2)^2}\sum_{\substack{ \xi_0\in\mathbb{Z}\backslash \{0\}\\ d_0^2l_0f^2\xi_0\ll X^{\frac12+\kappa} \\d_0\xi_0\ll  X^{\frac16+\kappa+\alpha}}}\xi_0^M.
\end{align*}
We now bound $S_G^{\nu\neq0,(1)}(X,\kappa,\alpha)$ and $S_G^{\nu\neq0,(2)}(X,\kappa,\alpha)$.
\begin{align*}
S_G^{\nu\neq0,(1)}(X,\kappa,\alpha)&\ll X^{-\frac M2}\sum_{X^{\frac12+\kappa}\gg d_0l_0f^2\gg X^{\frac14-\kappa}}\tfrac{(l_0f^2)^{M-2}}{d_0}\Bigl(1+\min\Bigl\{\tfrac{X^{\frac12+\kappa}}{d_0^2l_0f^2},\tfrac{X^{\frac16+\kappa+\alpha}}{d_0}\Bigr\}\Bigr)\\
&\ll X^{-\frac12+\kappa(M-1)}\log(X)+X^{-\frac M2}\sum_{ l_0f^2\ll X^{\frac12+\kappa}}(l_0f^2)^{M-2}\min\Bigl\{\tfrac{X^{\frac12+\kappa}}{l_0f^2},X^{\frac16+\kappa+\alpha}\Bigr\}\\
&\ll X^{-\frac12+\kappa(M-1)}\log(X).\tag{$i$}\label{fourierendest1}
\end{align*}
To bound $S_G^{\nu\neq0,(2)}(X,\kappa,\alpha)$ first note that since $d_0\xi_0\ll X^{\frac16+\kappa+\alpha}$ we have $d_0\ll X^{\frac16+\kappa+\alpha}$. Then,
\begin{align*}
S_G^{\nu\neq0,(2)}(X,\kappa,\alpha)&\ll \sum_{\substack{d_0\ll X^{\frac16+\kappa+\alpha}\\ X^{\frac12+\kappa}\gg d_0l_0f^2\gg X^{\frac14-\kappa}}}\tfrac{1}{d_0(l_0f^2)^2}\Bigl(1+\min\Bigl\{\Bigl(\tfrac{X^{\frac12+\kappa}}{d_0^2l_0f^2}\Bigr)^{M+1},\Bigl(\tfrac{X^{\frac16+\kappa+\alpha}}{d_0}\Bigr)^{M+1}\Bigr\}\Bigr)\\
&\ll \log(X)+\sum_{ \substack{d_0\ll X^{\frac16+\kappa+\alpha}\\l_0f^2\gg \lfloor\frac{X^{\frac14-\kappa}}{d_0}\rfloor}}\tfrac{1}{d_0(l_0f^2)^2}\min\Bigl\{\Bigl(\tfrac{X^{\frac12+\kappa}}{d_0^2l_0f^2}\Bigr)^{M+1},\Bigl(\tfrac{X^{\frac16+\kappa+\alpha}}{d_0}\Bigr)^{M+1}\Bigr\}\\
&\ll \log(X)+X^{(\frac16+\kappa+\alpha)(M+1)}(X^{\kappa-\frac14}+X^{\alpha-\frac13})\\
&\leq X^{(\frac16+\kappa+\alpha)(M+1)-\frac14+\kappa}.\tag{$ii$}\label{fourierendest2}
\end{align*}
Where, we used the assumption that $2\kappa+\alpha<\frac1{12}$ therefore $\frac14-\kappa<\frac13-\alpha$. Finally, substituting \eqref{fourierendest1} and \eqref{fourierendest2} into \eqref{fourierendest0} gives,
\begin{align*}
S_G^{\nu\neq0}(X,\kappa,\alpha)&\ll \|G \|_{{M,1}}X^{\frac{3-M}{2}}\log(X)(X^{-\frac12+\kappa(M-1)}\log(X)+X^{(\frac16+\kappa+\alpha)(M+1)-\frac14+\kappa})\\
&\ll \|G \|_{{M,1}}X^{\frac{17-4M}{12}+(M+1)(\kappa+\alpha)+\kappa}\log(X).
\end{align*}

\end{itemize}
\end{proof}

\begin{cor}\label{cors}Let $\kappa,\alpha>0$ such that $2\kappa+\alpha<\frac{1}{12}$. Then for every $\epsilon>0$ we have
\[S(X,\kappa,\alpha)=O\bigl(X^{\frac{11}{12}+\kappa+\alpha+\frac{11\epsilon}{6}}\bigr),\]
where the implied constant depends only on $k,\kappa,\alpha,$ and $\epsilon$.
\end{cor}

\begin{proof} We start with bounding each individual $T_k(n)$. Although there are better bounds (in particular, the Ramanujan conjecture $|T_k(n)|\leq n^{\kappa}$ is known in this case thanks to Deligne \cite{Deligne:1974aa}, and see \cite{Sarnak:2004aa} for an excellent survey of bounds towards the Ramanujan conjecture in general) in order to keep the proof self contained, and use only the trace formula we will refer to Theorem 1.1 of \cite{Altug:2015ab} which, translated to the setting of the current paper, tells us that $\tr(T_k(p))=O(p^{\frac k4})$ for any prime $p$. Then, by the recursion relations the $T_k(p^k)$'s satisfy (cf. (73) on pg.102 of \cite{Serre:1996aa}) it is straightforward to see  that $\tr(T_k(n))=O_{\epsilon}(n^{\frac{1}{4}+\epsilon})$.
 
 Using this it is now straightforward to prove the theorem by choosing a specific mollifier $G(x)$.  Let $\phi\in C_c^{\infty}((-1,1))$ such that it is normalized by $\int \phi=1$, and let $\bold{1}_{[\frac12,1)}$ be the characteristic function of the interval $[\frac12,1)$. Let $1\geq\delta\geq0$ set $\phi_{Y^\delta}(x)=Y^{1-\delta}\phi(xY^{1-\delta})$. Using $\phi_{Y^\delta}(x)$ define $G_{Y^\delta}(x):=\bold{1}_{[\frac12,1)}*\phi_{Y^\delta}(x)$, i.e.
\begin{equation}
G_{Y^\delta}(x)=Y^{1-\delta}\int\bold{1}_{[\frac12,1)}(x-y)\phi(yY^{1-\delta})dy.
\end{equation}
Then it is straightforward to see that 
\begin{equation}\label{moll}
 G_{Y^\delta}\left(\frac{x}{Y}\right)=\begin{cases}1& \frac Y2+Y^{\delta}\leq x\leq Y- Y^{\delta}\\ O(1)& |x-\frac Y2|< Y^{\delta}\,\, or \,\, |x-Y|< Y^{\delta}\\ 0 & otherwise \end{cases},
 \end{equation}
 and the 
\begin{equation}\label{moll2}
\|G_{Y^\delta}(x)\|_{{M,1}}=O(Y^{M(1-\delta)}).
\end{equation}
Then, by \eqref{moll} and the bound $\tr(T_k(n))=O(n^{\frac14+\epsilon})$ we get 
\begin{align*}
S(X,\kappa,\alpha)&=\sum_{j=1}^{\log(X)}\sum_{2^{j-1}\leq n<2^j}\sqrt{n}\sum_{\substack{X^{\frac12+\kappa}\gg lf^2\gg X^{\frac14-\kappa}}}\tfrac{1}{l^2f^3}\sum_{\substack{ \xi\in\mathbb{Z}\backslash \{0\}\\ lf^2\xi\ll X^{\frac12+\kappa}\\ \xi\ll X^{\frac16+\kappa+\alpha}}}Kl_{l,f}(\xi,n)I_{l,f}(\xi,n)\\
&\ll \sum_{j=1}^{\log(X)} \left\{S_{G_{2^{j\delta}}}(2^j,\kappa,\alpha)+2^{j\delta+\frac14+\epsilon}\right\},\tag{*}\label{moll3}
\end{align*}
where the implied constant depends only on $k,\delta$ and $\epsilon$. By Theorem \ref{thmsigma} for every $M>0$, the sub-sums, $S_{G_{2^{j\delta}}}(2^j,\kappa,\alpha)$, satisfy
\[S_{G_{2^{j\delta}}}(2^j,\kappa,\alpha)=O(\|G_{2^{j\delta}} \|_{{M,1}}j(2^j)^{\frac{17}{12}+2\kappa+\alpha-M(\frac13-\kappa-\alpha)}),\]
where the implied constant depends only on $k,\kappa,\alpha,$ and $M$. Substituting this bound in \eqref{moll3} and using \eqref{moll2} we get
\begin{align*}
S(X,\kappa)&=O\Bigl(\sum_{j=1}^{\log(X)}j2^{jM(1-\delta)}(2^j)^{\frac{17}{12}+2\kappa+\alpha-M(\frac13-\kappa-\alpha)}+2^{j\delta+\frac14+\epsilon}\Bigr)\\
&=O\bigl(X^{M(\frac23+\kappa+\alpha-\delta)+\frac{17}{12}+2\kappa+\alpha}\log(X)+X^{\delta+\frac14+\epsilon}\bigr),
\end{align*}
where the implied constant depends only on $k,\epsilon,M,$ and $\delta$. Finally choosing $\delta=\frac23+\kappa+\delta+\frac{5\epsilon}{6}$ and $M=\frac{2}{\epsilon}$ the corollary follows.

\end{proof}

\section{Local analysis}\label{seclocal}
In this section we will derive bounds on the Fourier transforms and character sums that appear after the Poisson summation on the $n$-sum.

\subsection{Archimedean analysis}

We begin with a technical lemma that will be useful for the rest of this section. For what follows let us fix two positive integers $l$ and $f,$ and let $X$ denote an independent parameter as in the previous sections. 

\begin{lemma}\label{archlem1}Let $\Phi\in\mathcal{S}(\mathbb{R}),\xi,\alpha\in \mathbb{Z},$ and $\nu\in\mathbb{Z}\backslash\{0\}$. Let,
\[V_{l,f,\alpha}(\xi,\nu,X):=\int G\left(\tfrac{y}{X}\right)y^{\frac{\alpha}{2}}\Phi\left(\tfrac{lf^2}{\sqrt{4y}\sqrt{1-x^2}}\right)e\left(\tfrac{-(x\xi\sqrt{4y}+y\nu)}{4lf^2}\right)dy.\]
Then, for any $G(x)\in C_c^{\infty}([\frac14,\frac54])$ and $M,N\in\mathbb{N}$ we have

\[V_{l,f,\alpha}(\xi,\nu,X)=O\left(\| G \|_{M,1}\tfrac{X^{1+\frac\alpha2}}{\nu^MX^{\frac M2}}\left[\left(\tfrac{lf^2}{\sqrt{X}}\right)^M+\xi^M\right]\left(\tfrac{\sqrt{X}\sqrt{1-x^2}}{lf^2}\right)^N\right),\]
where the implied constant depends only on $\Phi,\alpha,M,$ and $N$.
\end{lemma}

\begin{proof}First using the change of variables $y\mapsto Xy$ and then applying integration by parts $M$-times (keeping in mind that $G$ is compactly supported) gives,
\begin{equation}\label{ibp1}
V_{l,f,\alpha}(\xi,\nu,X)=X^{1+\frac{\alpha}{2}}\left(\tfrac{4lf^2}{-2\pi i X \nu}\right)^M\int \tfrac{d^M}{dy^M}\left\{G\left(y\right)y^{\frac{\alpha}{2}}\Phi\left(\tfrac{lf^2}{\sqrt{4Xy}\sqrt{1-x^2}}\right)e\left(\tfrac{-x\xi\sqrt{4Xy}}{4lf^2}\right)\right\}e\left(\tfrac{-Xy\nu}{4lf^2}\right)dy.
\end{equation}
The $M$'th derivative above is a combination of derivatives (of orders $\leq M$) of $G,$ $y^{\frac\alpha2},$ $\Phi\left(\frac{lf^2}{\sqrt{4Xy}\sqrt{1-x^2}}\right),$ and the exponential. Note that since $G(y)$ is compactly supported away from $y=0$ the negative powers of $y$ that appear in the derivative are bounded uniformly depending only on $M$ and cause no problem.  The only point we need to pay attention is the derivatives of $\Phi$. To that end, note that since $\Phi$ decays faster than any polynomial, for any $\beta_1,\beta_2\in\mathbb{N}$ we have
\[\left(\tfrac{lf^2}{\sqrt{X}\sqrt{1-x^2}}\right)^{\beta_1}\Phi^{(\beta_2)}\left(\tfrac{lf^2}{\sqrt{4Xy}\sqrt{1-x^2}}\right)\ll_{\Phi,N_1,\beta_1,\beta_2} \left(\tfrac{\sqrt{X}\sqrt{1-x^2}}{lf^2}\right)^{N}. \]
Using this bound we then get
\begin{equation}\label{ibp2}
\tfrac{d^M}{dy^M}\left\{G\left(y\right)y^{\frac{\alpha}{2}}\Phi\left(\tfrac{lf^2}{\sqrt{4Xy}\sqrt{1-x^2}}\right)e\left(\tfrac{-x\xi\sqrt{4Xy}}{4lf^2}\right)\right\}\ll \left(\tfrac{\sqrt{X}\sqrt{1-x^2}}{lf^2}\right)^N\left(1+\left(\tfrac{\xi\sqrt{X}}{lf^2}\right)^M\right),
\end{equation}
where the implied constant depends only on $\Phi, M,N,$ and $\| G \|_{M,1}$. Combining \eqref{ibp1} with \eqref{ibp2} finishes the proof.

\end{proof}

For the next corollary let us introduce the following notation,
\begin{equation}\label{fourier}
J_{l,f}(\xi,\nu,X):=\iint \sqrt{y}\,G\left(\tfrac{y}{X}\right)\theta_{\infty}(x)\left\{F\left(\tfrac{lf^2}{\sqrt{4y}\sqrt{1-x^2}}\right)+\tfrac{lf^2}{\sqrt{4y}\sqrt{1-x^2}}H\left(\tfrac{lf^2}{\sqrt{4y}\sqrt{1-x^2}}\right)\right\}e\left(\tfrac{-(x\xi\sqrt{4y}+y\nu)}{4lf^2}\right)dxdy.
\end{equation}

\begin{prop}\label{archprop}Let $\xi,\alpha\in \mathbb{Z},$ and $\nu\in\mathbb{Z}\backslash\{0\}$. Then, for any $G(x)\in C_c^{\infty}([\frac14,\frac54])$ and $M,N\in\mathbb{N}$ we have
\begin{align*}
J_{l,f}(\xi,0,X)&\ll \| G \|_1\tfrac{X^{\frac{N+3}{2}}}{(lf^2)^N\xi^{N+2}},\\
J_{l,f}(\xi,\nu,X)&\ll \| G \|_{M,1} \tfrac{X^{\frac{N-M+ 3}{2}}}{\nu^M(lf^2)^{N}}\left[\left(\tfrac{lf^2}{\sqrt{X}}\right)^M+\xi^M\right],
\end{align*}
where the implied constant depends only on $\theta_{\infty},  F,M,$ and $N$.
\end{prop}

\begin{proof} We first remark that because $G(y)$ is compactly supported away from $y=0$ we always have $y>0$. 

\begin{itemize}

\item $\nu=0$. By corollary 4.8. of \cite{Altug:2015ab}, for any $N\geq0$ and $y\neq0$, we have 
\[\int \theta_{\infty}(x)\left\{F\left(\tfrac{lf^2}{\sqrt{4Xy}\sqrt{1-x^2}}\right)+\tfrac{lf^2}{\sqrt{4Xy}\sqrt{1-x^2}}H\left(\tfrac{lf^2}{\sqrt{4Xy}\sqrt{1-x^2}}\right)\right\}e\left(\tfrac{-x\xi\sqrt{4Xy}}{4lf^2}\right)dx\ll \left(\tfrac{\sqrt{Xy}}{lf^2\xi}\right)^N\tfrac{1}{\xi^2}.\]
Using the change of variables $y\mapsto Xy$ and using the above bound gives,
\begin{align*}
J_{l,f}(\xi,0,X)&\ll \tfrac{X^{\frac{N+3}{2}}}{(lf^2)^N\xi^{N+2}}\int G(y)y^{\frac{N+1}{2}}dy\ll \| G \|_1 \tfrac{X^{\frac{N+3}{2}}}{(lf^2)^N\xi^{N+2}},
\end{align*}
where the implied constants depend only on $\theta_{\infty},F,$ and $N$. This finishes the proof of the case $\nu=0$.

\item $\nu\neq0$. Let,
\[J_{l,f}(\xi,\nu,X)=J^1_{l,f}(\xi,\nu,X)+J^2_{l,f}(\xi,\nu,X),\]
where
\begin{align*}
J^1_{l,f}(\xi,\nu,X)&:=\iint G\left(\tfrac{y}{X}\right)\sqrt{y}\theta_{\infty}(x)F\left(\tfrac{lf^2}{\sqrt{4y}\sqrt{1-x^2}}\right)e\left(\tfrac{-(x\xi\sqrt{4y}+y\nu)}{4lf^2}\right)dxdy,\\
J^2_{l,f}(\xi,\nu,X)&:=\tfrac{lf^2}{2}\iint G\left(\tfrac{y}{X}\right)\tfrac{\theta_{\infty}(x)}{\sqrt{1-x^2}}H\left(\tfrac{lf^2}{\sqrt{4y}\sqrt{1-x^2}}\right)e\left(\tfrac{-(x\xi\sqrt{4y}+y\nu)}{4lf^2}\right)dxdy,
\end{align*}

 By lemma 4.5 of \cite{Altug:2015ab} both $F$ and $H$ are in $\mathcal{S}(\mathbb{R})$. Moreover recall that $\theta_{\infty}(x)$ and $G(y)$ are both compactly supported. Therefore, the double integrals in $J^1_{l,f}(\xi,\nu,X)$ and $J^2_{l,f}(\xi,\nu,X)$ are both absolutely convergent, and hence we can interchange the order of integration in both. Doing so and using lemma \ref{archlem1} in the $y$-integrals (take $\alpha=1$ in $J^1_{l,f}(\xi,\nu,X)$ and $\alpha=0$ in $J^2_{l,f}(\xi,\nu,X)$) gives, for any $M,N_0,N_1\in\mathbb{N}$, 
\begin{align*}
J^1_{l,f}(\xi,\nu,X)&\ll \| G \|_{M,1}\tfrac{X^{\frac{N_0-M+ 3}{2}}}{\nu^M(lf^2)^{N_0}}\left[\left(\tfrac{lf^2}{\sqrt{X}}\right)^M+\xi^M\right]\int \theta_{\infty}(x)(1-x^2)^{\frac{N_0}{2}}dx,\\
J^2_{l,f}(\xi,\nu,X)&\ll \| G \|_{M,1}\tfrac{X^{1+\frac{N_1-M}{2}}}{\nu^M(lf^2)^{N_1-1}}\left[\left(\tfrac{lf^2}{\sqrt{X}}\right)^M+\xi^M\right]\int \theta_{\infty}(x)(1-x^2)^{\frac{N_1-1}{2}}dx.
\end{align*}
Finally choosing $N_0=N$ and $N_1=1+N$, which guarantees that the $x$-integrals converge, finishes the proof.
\end{itemize}
\end{proof}

\subsection{Non-Archimedean analysis}

Let us first introduce some notation that will be used throughout the calculations. Let $p$ be a prime. For any integer $A\in\mathbb{Z}$ let $v_p(A)$ denote the $p$-adic valuation of $A$. In what follows we will denote the ``$p$-part'' and the ``prime to $p$-part'' of $A$ respectively by $A_{(p)}$ and $A^{(p)}$. They are defined by,
\[A_{(p)}:=q^{v_p(A)}\hspace{0.5in}and\hspace{0.5in}A^{(p)}:=\tfrac{A}{A_{(p)}}.\]
For an integer $A\in\mathbb{Z}\backslash \{0\},$ let $rad(A)$ denote the radical of $A$. i.e.
\[rad(A)=\prod_{\substack{p\mid A\\p-prime}}p.\]
Finally, let us introduce the character sums that will be the focus of this section:
\begin{equation}\label{omega}
\omega_{l,f}(\xi,\nu):=\sum_{b\bmod 4lf^2}Kl_{l,f}(\xi,\nu)e\left(\tfrac{b\nu}{4lf^2}\right).
\end{equation}

\begin{lemma}\label{multlem}Let $k_1,k_2\in \mathbb{N}$. Then,
\[\omega_{l,f}(\xi,\nu)=\prod_{p}\omega_{p^{v_p(l)},p^{v_p(f)}}\left(((4lf^2)^{(p)})^{-1}\xi,((4lf^2)^{(p)})^{-1}\nu\right)\]
\end{lemma}

\begin{proof}The proof follows from the Chinese remainder theorem. The details are exactly the same as in lemma B.1 of \cite{Altug:2015ab}.

\end{proof}

Lemma \ref{multlem} reduces the calculation of $\omega_{l,f}(\xi,\nu)$ to $\omega_{p^{k_1},p^{k_2}}(\alpha,\beta),$ where $\alpha,\beta\in\mathbb{Z}$.

\begin{lemma}\label{charsumlem1}Let $p$ be an odd prime, $m\in\mathbb{N},$ and $\alpha,\beta\in\mathbb{Z}$. Then,
\[\sum_{a\bmod p^{m}}e\left(\tfrac{2a\alpha +a^2\beta}{p^{m}}\right)=\begin{cases}p^m& \beta=0,\,\, v_p(\alpha)\geq m\\ 0 & \beta=0,\,\, v_p(\alpha)< m\\p^{\frac{m+\min\{m,v_p(\delta)\}}{2}}\eta(p^{m-v_p(\delta)})\left(\tfrac{p^{-v_p(\delta)}\beta}{p^{m-v_p(\delta)}}\right)e\left(\tfrac{-\beta(\alpha_0\beta_0^{-1})^2}{p^{m}}\right)& \beta\neq0\end{cases},\]
where
\[\delta=\gcd(\alpha,\beta),\hspace{0.5in}\alpha=\delta\alpha_0,\hspace{0.5in}\beta=\delta\beta_0,\]
and for any $n\in \mathbb{N}$
\[\bar{\eta}(n)=\tfrac{1+i^n}{1+i}\]
($\bar{\eta}(n)$ denoting the complex conjugate of $\eta(n)$). Finally, we emphasize that if $v_p(\beta)> v_p(\gcd(\alpha,\beta))$  the right hand side is $0$ because of the appearance of the Jacobi symbol, $\left(\tfrac{p^{-v_p(\delta)}\beta}{p^{m-v_p(\delta)}}\right)$. 
\end{lemma}

\begin{proof}First, note that if $\beta=0$ the sum is a complete character sum over $a$ and is $0$ unless $v_p(\alpha)\geq m,$ in which case it is $p^{m}$. Also, the sum is trivially $p^m$ if $v_p(\delta)\geq m,$ so for the rest of the calculation we assume that $\beta\neq0$ and $v_p(\delta)< m$.

Let $\delta=\delta_0 p^{v_p(\delta)}$ and let $a=a_0+a_1p^{m-v_p(\delta)},$ where $a_0$ and $a_1$ are running modulo $p^{m-v_p(\delta)}$ and $p^{v_p(\delta)}$ respectively. Then,
\begin{equation}\label{gauss1}
\sum_{a\bmod p^{m}}e\left(\tfrac{2a\alpha +a^2\beta}{p^{m}}\right)=p^{v_p(\delta)}\sum_{a_0\bmod p^{m-v_p(\delta)}}e\left(\tfrac{\delta_0(2a_0\alpha_0+a_0^2\beta_0)}{p^{m-v_p(\delta)}}\right).
\end{equation}

\begin{itemize}
\item {\underline{\textit{Claim}}: } $v_p(\beta_0)>0\Rightarrow\eqref{gauss1}=0$. 

\begin{proof}
Suppose $v_p(\beta_0)>0$ and let $a_0=a_2+a_3p^{m-v_p(\delta)-v_p(\beta_0)}$., where $a_2$ and $a_3$ are running modulo $p^{m-v_p(\delta)-v_p(\beta_0)}$ and $p^{v_p(\beta_0)}$ respectively. Then, 
\[\eqref{gauss1}=p^{v_p(\delta)}\sum_{a_2\bmod p^{m-v_p(\delta)-v_p(\beta_0)}}e\left(\tfrac{\delta_0(2a_2\alpha_0+a_2^2\beta_0)}{p^{m-v_p(d)}}\right)\sum_{a_3\bmod p^{v_p(\beta_0)}}e\left(\tfrac{2\delta_0a_3\alpha_0}{p^{v_p(\beta_0)}}\right).\]
Since $\gcd(\alpha_0,\beta_0)=1,$ $v_p(\beta_0)>0\Rightarrow v_p(\alpha_0)=0$. We also have $v_p(2\delta_0)=0$ (recall that $p\neq2$). Therefore, the last sum over $a_3$ vanishes.
\end{proof}
\end{itemize}
 Furthermore, again by the claim above we can assume that $v_p(\beta_0)=0,$ otherwise \eqref{gauss1} is $0$. Then, 
\begin{align*}
\eqref{gauss1}&=p^{v_p(\delta)}\sum_{a_0\bmod p^{m-v_p(\delta)}}e\left(\tfrac{\delta_0\beta_0(2a_0\alpha_0\beta_0^{-1}+a_0^2)}{p^{m-v_p(\delta)}}\right)\\
&=p^{v_p(\delta)}\sum_{a_0\bmod p^{m-v_p(\delta)}}e\left(\tfrac{\delta_0\beta_0(a_0+\alpha_0\beta_0^{-1})^2-(\alpha_0\beta_0^{-1})^2)}{p^{m-v_p(\delta)}}\right)\\
&=p^{v_p(\delta)}e\left(\tfrac{-\beta(\alpha_0\beta_0^{-1})^2}{p^{m}}\right)\sum_{a_0\bmod p^{m-v_p(\delta)}}e\left(\tfrac{\delta_0\beta_0a_0^2}{p^{m-v_p(\delta)}}\right)\\
&=p^{\frac{m+v_p(\delta)}{2}}\eta(p^{m-v_p(\delta)})\left(\tfrac{\delta_0\beta_0}{p^{m-v_p(\delta)}}\right)e\left(\tfrac{-\beta(\alpha_0/\beta_0)^2}{p^{m}}\right).
\end{align*}
Note that in the last line we used the explicit calculation of the Gauss sum (cf. theorem 3.4 of \cite{Iwaniec:2004aa}).
\end{proof}

\begin{lemma}\label{charsumlem2}Let $p$ be an odd prime, $m\in \mathbb{N},$ and $\alpha,\beta\in\mathbb{Z}$. Then,
\[\sum_{b\bmod p^m}\left(\tfrac{b}{p^m}\right)e\left(\tfrac{b\beta}{p^m}\right)=\begin{cases}\phi(p^{m})& v_p(\beta)\geq m,\,\, m\equiv 0\bmod 2\\ -p^{m-1}& v_p(\beta)=m-1,\,\, m\equiv0\bmod 2\\ \left(\frac{p^{-v_p(\beta)}\beta}{p}\right)\eta(p)p^{m-\frac12}& v_p(\beta)=m-1,\,\, m\equiv 1\bmod 2\\ 0& otherwise\end{cases}\]

\end{lemma}

\begin{proof}Let $b=b_0+b_1p$. Then,
\begin{align*}
\sum_{b\bmod p^m}\left(\tfrac{b}{p^m}\right)e\left(\tfrac{b\beta}{p^m}\right)&=\sum_{\substack{b_0 \bmod p\\ b_1\bmod p^{m-1}}}\left(\tfrac{b_0}{p^m}\right)e\left(\tfrac{b_0\beta}{p^m}\right)e\left(\tfrac{b_1\beta}{p^{m-1}}\right)\\
&=\sum_{\substack{b_0 \bmod p}}\left(\tfrac{b_0}{p^m}\right)e\left(\tfrac{b_0\beta}{p^{m}}\right)\begin{cases}p^{m-1}& v_p(\beta)\geq m-1\\ 0& otherwise\end{cases}.
\end{align*}
A case by case calculation of the $b_0$-sum finishes the proof.

\end{proof}

\begin{prop}\label{propodd}Let $p$ be an odd prime, $k_1,k_2\in\mathbb{N},$ and $\alpha,\beta\in\mathbb{Z}$. Then, 
\[\omega_{p^{k_1},p^{k_2}}(\alpha,0)=\begin{cases}p^{k_1+2k_2}\phi(p^{k_1})& v_p(\alpha)\geq k_1+2k_2,\,\, k_1\equiv 0\bmod 2\\
0&otherwise\end{cases},\]
and for $\beta\neq0$ we have,
\[\omega_{p^{k_1},p^{k_2}}(\alpha,\beta)=\tau_p(k_1,k_2,\alpha,\beta)\begin{cases}\phi(p^{k_1})& v_p(\beta)\geq k_1,\,\, k_1\equiv 0\bmod 2\\ -p^{k_1-1}& v_p(\beta)=k_1-1,\,\, k_1\equiv 0\bmod 2\\ \left(\tfrac{-p^{-v_p(\beta)\beta}}{p}\right)\eta(p)p^{k_1-\frac12}& v_p(\beta)=k_1-1,\,\, k_1\equiv 1\bmod 2\\0&otherwise\end{cases},\]
where $\alpha_0,\beta_0,$ and $\delta$ are as in lemma \ref{charsumlem1}, and
\[\tau_p(k_1,k_2,\alpha,\beta):=^{\frac{k_1+2k_2+\min\{v_p(\beta),k_1+2k_2\}}{2}}\eta(p^{k_1+2k_2-v_p(\beta)})\left(\tfrac{p^{-v_p(\beta)}\beta}{p^{k_1+2k_2-v_p(\beta)}}\right)e\left(\tfrac{-\beta(\alpha_0\beta_0^{-1})^2}{p^{k_1+2k_2}}\right)\]
if $ v_p(\alpha)\geq \min\{v_p(\beta),k_1+2k_2\},$ and is $0$ otherwise.
\end{prop}

\begin{proof}First note that since $p$ is odd, the substitution $a\mapsto 2a$ does not change the sum and gives
\[\omega_{p^{k_1},p^{k_2}}(\alpha,\beta)=\sum_{\substack{a,b\bmod p^{k_1+2k_2}\\ a^2\equiv b\bmod p^{2k_2}}}\left(\tfrac{(a^2-b)/p^{2k_2}}{p^{k_1}}\right)e\left(\tfrac{2a\alpha+b\beta}{p^{k_1+2k_2}}\right).\]
For each $a \bmod\, p^{k_1+2k_2}$ the elements $b\bmod\, p^{k_1+2k_2}$ satisfying $a^2\equiv 4b\bmod \,p^{2k_2}$ can be parametrized by $b=a^2+b_0p^{2k_2},$ where $b_0$ is running $\bmod\, p^{k_1}$. Note that with this parametrization $a^2-b\equiv -b_0p^{2k_2} \bmod \,p^{k_1+2k_2}$. Using this observation we get,
\begin{align*}
\omega_{p^{k_1},p^{k_2}}(\alpha,\beta)&=\sum_{\substack{a \bmod p^{k_1+2k_2}\\ b_0 \bmod p^{k_1}}}\left(\tfrac{-b_0}{p^{k_1}}\right)e\left(\tfrac{b_0\beta}{p^{k_1}}\right)e\left(\tfrac{2a\alpha+a^2\beta }{p^{k_1+2k_2}}\right).
\end{align*}
The proposition now follows from lemmas \ref{charsumlem1} and \ref{charsumlem2} (One just needs to keep in mind that in lemma \ref{charsumlem1} the sum vanishes unless $v_p(\alpha)\geq \min\{m,v_p(\beta)\}$.).
\end{proof}

\begin{cor}\label{corodd}Let $p$ be an odd prime, $k_1,k_1\in\mathbb{N},$ and $\alpha,\beta\in\mathbb{Z}$. Then,
\[\omega_{p^{k_1},p^{k_2}}(\alpha,\beta)\ll\begin{cases}p^{k_1+k_2}\sqrt{\gcd(p^{k_1+2k_2},\beta)\gcd(p^{k_1},\beta)}& p^{k_1-1}\mid \beta,\,\, \gcd(p^{k_1+2k_2},\beta)\mid \alpha \\ 0& otherwise\end{cases},\]
where the implied constant is absolute. 
\end{cor}

\begin{proof}First assume $\beta\neq0$. Then, by the second statement of proposition \ref{propodd} we know that $\omega_{p^{k_1},p^{k_2}}(\alpha,\beta)$ is the product of two terms (the first is $\tau_p(k_1,k_2,\alpha,\beta),$ and the second is in the braces). The second one of those terms is $0$ unless $p^{k_1-1}\mid \beta,$ in which case it is $O(p^{\frac{k_1}{2}}\sqrt{\gcd(\beta,p^{k_1})}),$ where the implied constant is absolute. On the other hand, the first term satisfies the following bound:
\begin{align*}
\tau_p(k_1,k_2,\alpha,\beta)&\ll \begin{cases}p^{\frac{k_1+2k_2}{2}}\sqrt{\gcd(p^{k_1+2k_2},\beta)}& v_p(\alpha)\geq \min\{v_p(\beta),k_1+2k_2\}\\ 0& otherwise\end{cases}. 
\end{align*}
These two bounds imply the corollary in the case $\beta\neq0$. For $\beta=0,$ the bound in the statement is $O(p^{2k_1+2k_2})$ if $v_p(\alpha)\geq k_1+2k_2$ and $0$ otherwise. A comparison of this bound with the first statement of proposition \ref{propodd} finishes the proof.
\end{proof}

\begin{lemma}\label{lemeven} Let $k_1,k_1\in\mathbb{N}$ and $\alpha,\beta\in\mathbb{Z}$. Then,
\[\omega_{2^{k_1},2^{k_2}}(\alpha,\beta)\ll\begin{cases}2^{k_1+k_2}\sqrt{\gcd(2^{k_1+2k_2},\beta)\gcd(2^{k_1},\beta)}& 2^{k_1-1}\mid \beta,\,\, \gcd(2^{k_1+2k_2},\beta)\mid \alpha \\ 0& otherwise\end{cases},\]
where the implied constant is absolute. 
\end{lemma}

\begin{proof}The proof of this case follow the proofs of proposition \ref{propodd} and corollary \ref{corodd} verbatim. One just needs to take into account the extra condition $\frac{a^2-4b}{p^{2k_2}}\equiv 0,1 \bmod 4$ and recall that the Kronecker symbol $\left(\frac{\cdot}{2}\right)$ is periodic $\bmod\,8$. These do not bring any new ingredients to the proof but rather a delicate case by case analysis whose details we leave to the reader. 
\end{proof}

\begin{cor}\label{omegacor}Let $l,f\in\mathbb{Z}_{>0}$ and $\xi,\nu\in \mathbb{Z}$. Then,
\[\omega_{l,f}(\xi,\nu)\ll\begin{cases}\log(lf)lf\sqrt{\gcd(lf^2,\nu)\gcd(l,\nu)}& \frac{l}{rad(l)}\mid \nu,\,\, \gcd(lf^2,\nu)\mid \xi \\ 0& otherwise\end{cases},\]
where the implied constant is absolute. 
\end{cor}

\begin{proof}
By lemma \ref{multlem} it is enough to bound $\omega_{p^{k_1},p^{k_2}}(\alpha_p,\beta_p),$ where $\alpha_p=((4lf^2)^{(p)})^{-1}\xi$ and $\beta_p=(4lf^2)^{(p)})^{-1}\nu,$ and to bound $\prod_{p\mid lf^2}O(1)$. The bound on the character sums follow from corollary \ref{corodd} and lemma \ref{lemeven}. Finally, by the prime number theorem we have $\prod_{p\mid lf^2}O(1)=O(\log(lf^2))=O(lf)$. The corollary follows.
\end{proof}

\bibliographystyle{alpha}

\end{document}